\documentclass[11pt]{amsart}
\usepackage{amssymb}
\usepackage{amsmath}
\usepackage{amsthm}
\usepackage{amscd}
\usepackage{amsfonts}
\usepackage{tikz}
\usepackage{tikz-qtree}
\usepackage{tikz-cd}
\usepackage{mathrsfs}
\usepackage[all]{xy}
\usepackage{faktor}
\usepackage{graphicx}
\usepackage{blkarray}
\usepackage{amsmath}
\usepackage{lineno}
\usepackage{tikz}
\usepackage{xparse}
\usepackage{ stmaryrd }
\usepackage{pgf}
\usepackage{enumerate}
\usetikzlibrary{calc,fit,matrix,arrows,automata,positioning}
\usetikzlibrary{arrows.meta}
\usepackage{mathtools}
\usepackage{amsmath}
\usepackage{amsmath}
\usepackage{amscd}
\usepackage{amsfonts}
\usepackage{tikz}
\usepackage{tikz-cd}
\usepackage{mathrsfs}
\usetikzlibrary{calc}
\usepackage[all]{xy}
\usepackage{faktor}
\usepackage{enumitem}
\usepackage{url}
\usetikzlibrary{matrix}
\usepackage{tikz-qtree}
\usepackage{tikz-cd}
\usepackage{mathrsfs}
\usetikzlibrary{trees}
\usetikzlibrary{calc}
\usepackage{graphicx}
\usepackage{graphics}
\usepackage{xfrac}
\usepackage{float}
\usepackage{array}
\usepackage{xcolor}

\DeclareMathOperator{\MAX}{max}

\DeclareMathOperator{\degree}{deg}
\DeclareMathOperator{\Mat}{Mat}

\DeclareMathOperator{\Image}{im}
\DeclareMathOperator{\Rank}{rank}
\DeclareMathOperator{\Freeclass}{\text{FreeClass}}

\DeclareMathOperator{\image}{im}
\DeclareMathOperator{\kernel}{ker}
\DeclareMathOperator{\infimum}{inf}

\DeclareMathOperator{\WLOG}{WLOG}
\newtheorem{theorem}{Theorem}[section]
\newtheorem{lemma}{Lemma}[section]
\newtheorem{proposition}{Proposition}[section]
\newtheorem{corollary}{Corollary}[section]
\newtheorem{conjecture}{Conjecture}[section]
\theoremstyle{definition}
\newtheorem{definition}{Definition}[section]

\newtheorem{remark}{Remark}[section]

\newcommand{\RN}[1]{
  \textup{\uppercase\expandafter{\romannumeral#1}}%
}
\newcommand{\RowNumber  }{\mathcal{R}}
\newcommand{\ColNumber  }{\mathcal{C}}
\usetikzlibrary{matrix,backgrounds}
\pgfdeclarelayer{myback}
\pgfsetlayers{myback,background,main}

\tikzset{mycolor/.style = {line width=1bp,color=#1}}%
\tikzset{myfillcolor/.style = {draw,fill=#1}}%

\NewDocumentCommand{\highlight}{O{blue!40} m m}{%
\draw[mycolor=#1] (#2.north west)rectangle (#3.south east);
}

\newcommand{\Homology}{\operatorname{H}}

\newcommand{\dimension}{\operatorname{dim}}

\usepackage{chessfss}

\newlength{\symsize}
\newlength{\boardwidth}
\setboardfontsize{\symsize}

\calclayout

\begin{document}

\title{The rank 8 case of a conjecture on square-zero upper triangular matrices}

\author{BERR\.{I}N \c SENT\" URK  }

\address{Department of Mathematics, TED
University, Ankara, 06420, Turkey.}

\email{ berrin.senturk@tedu.edu.tr }

\thanks{The author was supported by the Fellowship Program for Abroad Studies 2219 by
the Scientific and Technological Research Council of Turkey (T\"{U}B\.{I}TAK)}



\begin{abstract}
Let $A$ be the polynomial algebra in $r$ variables with coefficients in an algebraically closed field $k$.
 When the characteristic of $k$ is $2$, Carlsson
\cite{Carlsson1986} conjectured that any $\mathrm{dg}$-$A$-module that is free of rank $N$ as an $A$-module and whose homology is
nontrivial and finite dimensional as a $k$-vector space satisfies $N\geq 2^r$.
In this paper, we
examine a stronger conjecture concerning varieties of square-zero upper triangular
$N\times N$ matrices.
Stratifying these varieties via Borel orbits, we show that the
stronger conjecture holds when $N = 8$ without any restriction on the characteristic of $k$.
This result also verifies that if $X$ is a product of $3$ spheres of any dimensions, then
the elementary abelian $2$-group of rank $4$ cannot act freely on $X$.
\end{abstract}
\maketitle

\section{Introduction}
\label{section:Introduction}
The long-standing Rank Conjecture states that if $(\mathbb{Z}/p\mathbb{Z})^r$ acts freely and cellulary
 on a finite CW-complex $X$ that is homotopy equivalent to ${S^{n_1}\times \ldots \times S^{n_m}}$, then  $r\leq m$. There are many partial results in several special cases. In
the equidimensional case $n:=n_1=\ldots=n_m$, Carlsson \cite{Carlsson1982},
Browder \cite{Browder}, and Benson-Carlson
\cite{BensonCarlson} gave a proof under the assumption that the induced
action on homology is trivial. Without the homology assumption, the
equidimensional conjecture was proved by Conner \cite{Conner} for $m=2$,
Adem-Browder \cite{AdemBrowder} for $p\neq2$ or $n\neq 1,3,7$, and Yal\c{c}{\i}n \cite{EYgroupact} for $p=2$ and $n=1$. In the
non-equidimensional case, the
conjecture was verified by Smith \cite{Smith} for $m=1$, Heller \cite{Heller} for
$m=2$,  Carlsson \cite{Carlsson2} for $p=2$ and $r=3$, Refai \cite{Refai1} for $p=2$ and $r=4$,  Hanke \cite{Hanke} for $p$ large relative to the dimension of the product of spheres, and
Okutan-Yal\c{c}{\i}n \cite{OkutanYal} for products in which the average of the
dimensions
is sufficiently large compared to the differences between them. The general case for
$r\geq 5$ is
still open.

A more general conjecture, known as Carlsson's Rank Conjecture \cite[Conjecture~$\RN{1}.3$]{Carlsson1986} states that if $(\mathbb{Z}/p\mathbb{Z})^r$ acts freely on a finite nonempty CW-complex $X$, then
$\sum_i\dim_{\mathbb{F}_p} \Homology_i(X;\mathbb{Z}/p\mathbb{Z})$ is at least $2^r$. Carlsson also states an algebraic analogue of the conjecture  \cite[Conjecture~$\RN{2}.2$]{Carlsson1986}, which is even stronger:
 	If $C_*$ is a finite free $\overline{\mathbb{F}}_p(\mathbb{Z}/p\mathbb{Z})^r$-chain complex with non-zero homology, then $\dim_{\overline{\mathbb{F}}_p}	\Homology_*(C_*)\geq 2^r$.
 However, Iyengar-Walker \cite{Iyengar_Walker_2018} disproved the algebraic conjecture when $p\neq 2$ and $r\geq 8$. Even so, it remains open for $p=2$. Moreover, the Iyengar-Walker counterexamples cannot be realized topologically \cite{Marc}, so the topological version of Carlsson's Rank Conjecture is still open for all primes.

 Let $R$ be a graded ring. A pair $(M,\partial)$ is a \emph{differential graded $R$-module}, or simply \emph{$\mathrm{dg}$-$R$-module}, if $M$ is a graded right $R$-module and
$\partial$ is an $R$-linear endomorphism of $M$ of degree $-1$ that satisfies $\partial^2=0$. A $\mathrm{dg}$-$R$-module is \emph{free} if the underlying $R$-module is free.

 \indent When $k$ is an algebraically closed field of characteristic $2$, Carlsson in \cite{Carlssonbeta} and \cite{Carlsson1986} showed that the algebraic analogue of the conjecture is equivalent to a conjecture in commutative algebra:
 \begin{conjecture}\label{conjecturePoly}
 Let $k$ be an algebraically closed field and $A=k[y_1,\ldots,y_r]$ the polynomial
ring over $k$ in the variables $y_1,\ldots,y_r$ of degree $-1$. If $M$ is a free,
 finitely generated $\mathrm{dg}$-$A$-module, $\Homology_*(M)\neq 0$, and
 $\dim_k \Homology_*(M)< \infty$, then  $N:=\Rank_A M\geq 2^r$.
 \end{conjecture}
  For $r\leq 3$, the above conjecture was verified  by Carlsson \cite{Carlsson2} when the characteristic of $k$ is $2$, by Avramov, Buchweitz, and Iyengar \cite{Avramov} for regular rings without any restriction on the characteristic of $k$. An analogue of the conjecture for the rational numbers was proved by Allday and Puppe \cite{AlldayPuppeEski}.
   Recently, VandeBogert and Walker \cite{KellerWalker} proved that the Total rank conjecture related to the Betti numbers of $M$ holds for the rings of characteristic two.

 In \cite{Berrin}, \"{U}nl\"{u} and the author  stated another conjecture from the perspective of algebraic geometry:
\begin{conjecture}\label{conjpaperRC}
Let $k$ be an algebraically closed field, $r$ a positive
integer, ${N=2n}$ an even positive integer, and $d:=(d_1,d_2,\dots ,d_N)$ an $N$-tuple of  nonincreasing integers. Then define $V(d,n)$ as the
weighted quasi-projective variety of rank $n$ square-zero upper triangular $N\times N$
matrices $(x_{ij})$ under the equivalence relation $({\lambda}^{\mathrm{d_i-d_j+1}}x_{ij})\sim (x_{ij})$ for all $\lambda$ in the unit group $k^{\ast}$.
Assume that there exists a
nonconstant morphism $\psi $ from the projective variety $\mathbb{P}^{r-1}_k$ to $V(d,n)$.
Then $N\geq 2^r$.
\end{conjecture}
We proved that Conjecture \ref{conjpaperRC} implies Conjecture \ref{conjecturePoly} in \cite[Theorem~$1$]{Berrin}. The kernel of the idea came from Carlsson's work in \cite{Carlssonbeta} and \cite{Carlsson2}. The key point is that every free $\mathrm{dg}$-$A$-module $(M,\partial)$ is quasi-isomorphic to a free $\mathrm{dg}$-$A$-module $(\tilde{M},\tilde{\partial})$ such that the differential of $\tilde{M}\otimes_A k$ is the zero map. This module $\tilde{M}$ is called a minimal module. Let $N$ be the rank of $\tilde{M}$ over $A$. Then there exists a homogeneous basis for $\tilde{M}$ so that the boundary $\tilde{\partial}$ can be represented by an upper triangular $N\times N$ matrix whose entries are homogeneous polynomials in $A$. If $\mathbf{m}$ is any maximal ideal of $A$ other than $(y_1,\ldots,y_r)$, then $\Homology_*(M\otimes_A A/\mathbf{m})=0$ since $\dim_k \Homology_*(M)< \infty$ (see \cite[$\S 1$ Proposition~$8$]{Carlsson2}). Moreover, for any finite-dimensional module $(M,\partial)$ over $k$, $\Homology_*(M)=0$  if and only if $\dim_kM=2 \Rank_k \partial$. To work in a uniform setting, we define a new polynomial ring $S$ by replacing our indeterminates with $x_i$ such that $\degree(x_i)=1$. Because of the differential of $M$, we focus on square-zero strictly upper triangular matrices whose entries are homogeneous polynomials $p_{ij}$'s in $S$ of degree $d_i-d_j+1$.

In our previous work \cite{Berrin}, some of the results  obtained by  the computer algebra program GAP for small dimensional cases of the conjecture led us to extend the conjecture for matrices of a certain form: Those with $\ColNumber$ leading zero columns and $\RowNumber$ ending zero rows. Hence the most general conjecture states that if the value of $(p_{ij})$ at every point in the image of $\psi $ is $0$ whenever $i\geq N-\RowNumber  +1$ or $j\leq \ColNumber  $, then $N\geq 2^{r-1}(\RowNumber  + \ColNumber )$. Since we are concerned only with strictly upper triangular matrices, we have $\RowNumber\geq 1$ and $\ColNumber\geq 1$. Clearly, in this situation we have $N\geq 2^r$.

Now suppose that the morphism $\psi$ in the above conjecture is represented by the matrix $\mathbf{\mathbf{\mathbf{D}}}$, so that $\mathbf{\mathbf{D}}$ is
 given by $p_{ij}$ coordinate-wise. We may also consider $\mathbf{D}$ as a differential morphism $\mathbf{D}:S^N\rightarrow S^N$ of degree $1$. Define
$\Homology(\mathbf{D})=\kernel(\mathbf{D})/ \Image(\mathbf{D})$. In \cite[Conjecture~$3$]{Berrin}, we considered a more general conjecture in terms of matrices. In this paper, we add the condition of finite and non-zero homology to the conjecture in order to use Theorem \ref{Theoremrandl}, and assert the following conjecture:

\begin{conjecture}\label{Berrinconjecture}
Let $k$ be an algebraically closed field and $S:=k[x_1,\ldots, x_r]$ the polynomial ring in $r$ variables with $\degree(x_i)=1$. Assume that $n$, $r$, $\RowNumber$, $\ColNumber$ are positive integers, $N:=2n$, and
$\mathbf{\mathbf{D}}=( p_{ij})\in \Mat_{N\times N}(S)$. If
\begin{enumerate}[label=\arabic*)]
\item $\mathbf{D}$ is strictly upper triangular,
\item $\mathbf{D}^2=0$,
\item $0< \dimension_k \Homology(\mathbf{D})< \infty$,
\item for all $i$ and $j$, we have $p_{ij}(0,\ldots,0)=0$ (i.e., each constant term is $0$),
\item for all $(a_1,\ldots,a_r)\in k^r-\{(0,\ldots,0)\}$, we have $\Rank(\mathbf{D}(a_1,\ldots,a_r))=n$,
\item there exists an $N$-tuple of nonincreasing integers $(d_1,\ldots,d_{N})$ such that for all $i$ and $j$, $p_{ij}$ is a homogeneous polynomial in $S$ of degree $d_i-d_j+1$,
\item when  $j\leq \ColNumber$  or  $i\geq N-\RowNumber +1$, we have $p_{ij}=0$,
\end{enumerate}
then $N\geq 2^{r-1}\left( \RowNumber + \ColNumber \right)$.
\end{conjecture}

When $N<8$, \"{U}nl\"{u} and the author have already proved Conjecture $3$ (and thus Conjecture \ref{Berrinconjecture}) in \cite[Theorem~$2$]{Berrin}.
Also, we verified the conjecture for $r\leq 2$ in \cite[Theorem~$6$]{Berrin}. However,
Iyengar and Walker's construction in \cite[Proposition~$2.1$]{Iyengar_Walker_2018} was used in \cite[Example~$3.2.4$]{berrinPhD} to form
an explicit counterexample to Conjecture \ref{Berrinconjecture} when the characteristic of $k$ is odd and $r=8$. Moreover, for any characteristic of $k$, we have a counterexample when $N=12$. If we extend the idea in \cite[Example~$0.4$]{Daniel} for $r=3$, we have the $12 \times 12$ matrix $\mathbf{D}$ with $\RowNumber=\ColNumber=2$, so the conjecture fails.

We state a new conjecture implied by Conjecture \ref{Berrinconjecture} in order to eliminate the counterexample that occurs when $r=3$ and $N=12$:
\begin{conjecture}\label{ENSONBerrinconjecture}
Let $k$ be an algebraically closed field and $S:=k[x_1,\ldots, x_r]$ the polynomial ring in $r$ variables with $\degree(x_i)=1$. Assume that $n$, $r$, $\RowNumber$, $\ColNumber$ are positive integers, $N:=2n$, and
$\mathbf{\mathbf{D}}=( p_{ij})\in \Mat_{N\times N}(S)$. If
\begin{enumerate}[label=\arabic*)]
\item $\mathbf{D}$ is strictly upper triangular,
\item $\mathbf{D}^2=0$,
\item $0< \dimension_k \Homology(\mathbf{D})< \infty$
\item for all $i$ and $j$, we have $p_{ij}(0,\ldots,0)=0$ (i.e., each constant term is $0$),
\item for all $(a_1,\ldots,a_r)\in k^r-\{(0,\ldots,0)\}$, we have $\Rank(\mathbf{D}(a_1,\ldots,a_r))=n$,
\item there exists an $N$-tuple of nonincreasing integers $(d_1,\ldots,d_{N})$ such that for all $i$ and $j$, $p_{ij}$ is a homogeneous polynomial in $S$ of degree $d_i-d_j+1$,
\item when  $j\leq \ColNumber$  or  $i\geq N-\RowNumber +1$, we have $p_{ij}=0$,
\end{enumerate}
then $N\geq 2^{r-1}\left(\MAX(\RowNumber , \ColNumber)+1 \right)$.
\end{conjecture}
The following Remark is given by the result in \cite[Lemma~$3.2.2$]{berrinPhD} obtained by \cite[Theorem~$6.5.7$]{HoffmanKunze} and a theorem of McCoy as stated in \cite{Shapiro}.
\begin{remark}\label{BerrinENsonRemark}
Let $k$ be an algebraically closed field, $A$ the polynomial ring in $r$ variables with coefficients in $k$. Suppose that $(M, \partial)$ is a free $\mathrm{dg}$-$A$-module with $0< \dimension_k \Homology(M)< \infty$. Then there exists a free $\mathrm{dg}$-$A$-module $(\tilde{M},\tilde{\partial})$ such that \begin{enumerate}[label=\arabic*)]
\item $0< \dimension_k \Homology(\tilde{M})< \infty$,
\item $(\tilde{M},\tilde{\partial})$ admits a free flag decomposition (see Definition \ref{deffreeflag}),
\item $\Rank_A(\tilde{M})\leq \Rank_A(M)$,
\item $(\tilde{M},\tilde{\partial})$ is minimal.
\end{enumerate}
\end{remark}
Note that by Remark \ref{BerrinENsonRemark}, Conjecture \ref{ENSONBerrinconjecture} implies Conjecture \ref{conjecturePoly}.

The main result of this paper holds without any restriction on the characteristic of the ground field $k$.
\begin{theorem}\label{Mainresulttheorem}$\mathrm{[Theorem~ \ref{MainTheorem}]}$
Conjecture \ref{Berrinconjecture} holds for $N=8$.
\end{theorem}

The proof of Conjecture \ref{Berrinconjecture} for $N=8$ also works for Conjecture \ref{ENSONBerrinconjecture}. The essential tool in the proof of Theorem \ref{Mainresulttheorem} is the stratification of the varieties of square-zero upper triangular matrices obtained by considering the conjugation action of a Borel subgroup of $GL_N(k)$. As in our previous work \cite[Section~$3.4$]{Berrin}, this stratification can be applied to the subvarieties that contain the image of the morphism defined by the matrix $\mathbf{D}$ in Conjecture \ref{Berrinconjecture}.

Note that the result ${N=8\geq 2^{r-1}\left(\RowNumber + \ColNumber \right) \geq 2^r}$ means that $r\geq 4$ implies $N \neq 8$. Together with the previous result that Conjecture \ref{Berrinconjecture} holds for $N<8$, one obtains that $r\geq 4$ implies $N>8$.  Now, consider a weaker version of Conjecture \ref{conjecturePoly}
for $k=\overline{\mathbb{F}}_2$, that is, if $M$ is a free,
 finitely generated $\mathrm{dg}$-$A$-module with $0<\dim_k \Homology_*(M)<\infty$, then  $N:=\Rank_A M\geq 2^{r-1}+1$. Note that by the K\"{u}nneth Formula even the weaker version of
 Conjecture \ref{conjecturePoly} implies Rank Conjecture for the $p=2$ case.
 \begin{corollary}\label{Corollary1}
Let $k$ be an algebraically closed field of characteristic $2$ and $A=k[y_1,\ldots,y_r]$. If  $r\geq 4$, and  $M$ is a free, finitely generated $dg$-$A$-module with
 ${0< \dim_k \Homology_*(M)<\infty}$, then  $N:=\Rank_A M>8$.
\end{corollary}
In particular, this result proves that the Rank Conjecture is true for product of three spheres of any dimensions:
\begin{corollary}\label{Corollary2}
Let $X$ be a finite CW-complex homotopy equivalent to
 $ S^{n_1}\times S^{n_2}\times \ldots \times S^{n_m}$. If $m\leq 3$, then $(\mathbb{Z}/2)^r$  cannot act freely and cellularly on $X$ for $r\geq 4$.
\end{corollary}
Therefore, Theorem \ref{Mainresulttheorem} gives an alternative proof of the result already given by Refai in \cite[Theorem~$4.5$]{Refai1}.

A simplifying observation is:
\begin{remark}\label{remarkvariable}
Let $\mathbf{D}$ be a matrix satisfying all six conditions of Conjecture \ref{Berrinconjecture} with ${r}$ variables. Then by evaluating some of  variables at $0$ (of course not letting all entries be identically zero), one can obtain a matrix that satisfies all six conditions in the conjecture with fewer variables. Hence
if there is no matrix that satisfies all six conditions in Conjecture \ref{Berrinconjecture} for $r$ variables, then the same is true for all matrices with more than $r$ variables.
\end{remark}

\section{Free flags}
The connection between various earlier algebraic versions of the conjecture was also
 explored by Avramov, Buchweitz, and Iyengar; and Conjecture \ref{conjecturePoly} was extended to local rings in \cite[Conjecture~$5.3$]{Avramov}.
In this section, we recall some earlier results given by \cite{Avramov} in a special case in order to prove some cases of our main result.

We are concerned with differential modules that admit a filtration compatible with their differentials:
\begin{definition}\cite[~$2.1$-$2.2$]{Avramov}\label{deffreeflag}
A \emph{differential module} over a ring $R$ is a pair $(F,d)$, where $F$ is an $R$-module and $d$ is an $R$-linear endomorphism of $F$ satisfying $d^2=0$. The \emph{homology} of $F$ is the $R$-module $\Homology(F)=\kernel(d)/ \image(d)$. A \emph{ free flag} of $F$ is a chain of $R$-submodules
\[
 0=F^{-1}\subseteq F^0\subseteq F^1\subseteq \ldots \subseteq F^l=F
 \]
 such that $F^i/F^{i-1}$ is a free $R$-module of finite rank and $d(F^i)\subseteq F^{i-1}$ for all $i$.
 We say that $F$ has a free flag with $(l+1)$-folds. If $F$ admits such a flag, we say that $F$ has free class at most $l$ and write $\Freeclass_R F\leq l$.
In other words, $$\Freeclass_R F=\infimum\{\ l\in \mathbb{N}\ | \ F \text{ admits a free flag with } (l+1) \text{-folds}\}.$$ We set $\Freeclass_R F=\infty$ if and only if $F$ admits no finite free flag.
\end{definition}

\begin{theorem}\cite[Theorem~$3$.$1$]{Avramov}\label{Avramov}
Let $R$ be a local ring, $F$ a differential $R$-module, and $\mathbf{D}$ a retract of $F$ such that the $R$-module $\Homology(\mathbf{D})$ has non-zero finite length. When $R$ has a big Cohen-Macaulay \cite[Section~$3.2$]{Avramov} module one has:
$$ \Freeclass_RF\geq \mathrm{dim}R.$$
\end{theorem}
\begin{theorem}\label{Theoremrandl}
Let $S:=k[x_1,\ldots, x_r]$ be the polynomial ring in $r$ variables with coefficients in the algebraically closed field $k$.
Suppose that $(F,d)$ is a differential graded $S$-module, and $\Homology(F)$ is  finite and non-zero. If $F$ has a free differential flag, i.e., $\Freeclass_S F\leq l$, then $l \geq r$.
\end{theorem}
\begin{proof}
Let $R$ be the localization of $S=k[x_1,\ldots,x_r]$ at the ideal $(x_1,\ldots,x_r)$.
Since $(x_1,\ldots,x_r)$ is one of the maximal ideals of $S$, we have the equality of Krull dimensions $\dim R=\dim S=r$.
  Note that big Cohen-Macaulay modules exist when our ring contains a field as a subring \cite{hochster}. Since ${0 \subseteq F^0\otimes_S R \subseteq  \ldots \subseteq F^l\otimes_S R=F\otimes_S R}$ is a flag,
 Theorem \ref{Avramov} implies that $ l\geq \Freeclass_R F\otimes_S R \geq r$.
\end{proof}
Note that Theorem \ref{Theoremrandl} verifies  Carlssson's earlier result in \cite[Theorem~$\RN{1}.16$]{Carlssonbeta}: If $S={\mathbb{F}}_2[x_1,\ldots,x_r]$ and $M$ is a free, finitely generated $\mathrm{dg}$-$S$-module with ${\Homology}_*(M)\neq 0$ and
 $\dim_k {\Homology}_*(M)<\infty$, then $\Freeclass_S M\geq r$ (see \cite[Remark~$3.4$]{Avramov}).
\begin{lemma}\cite[Lemma~$5.8$]{Avramov}\label{lemmaAvramov}
Let $S=k[x_1,\ldots,x_r]$ and $(F,d)$ be a finitely generated differential graded $S$-module. If $ \Freeclass_S F=l< \infty$, then we have $\Rank_SF\geq 2l$.
Moreover, if $F$ admits a free flag with $F^l=F$, then we have
\begin{itemize}
\item[$(a)$]$d(F^i)\nsubseteq F^{i-2} \text{ for } i=1,\ldots,l \, \text{ and}$
\item[$(b)$]
$\Rank_S(F^i/F^{i-1})\geq
 \begin{cases}
      1 & i=0 \text { or } i=l \\
      2 & i=1,\ldots, l-1 \\
   \end{cases}.$
\end{itemize}

\end{lemma}
\begin{remark}\label{remark2}
Given a matrix $\mathbf{D}$ that satisfies all six conditions of Conjecture \ref{Berrinconjecture}, there exists a corresponding free flag. More precisely, let $\mathbf{D}$ be a square-zero upper triangular block matrix of type $t=(t_0,t_1,\ldots,t_l)$, that is,

\begin{equation*}
   \mathbf{D} = \begin{tikzpicture}[baseline={(m.center)}]
           \matrix [matrix of math nodes,left delimiter={[},right delimiter={]}] (m)
           {
0 &\ldots& 0 &\scriptscriptstyle{p_{1,(t_0+1)}}&\ldots&\scriptscriptstyle{p_{1,(t_0+t_1)}}&\ldots& &\scriptscriptstyle{ p_{1,(N-t_l+1)}} &\ldots& \scriptscriptstyle{p_{1,N}} \\
\vdots &      &\vdots&\vdots         &      &   \vdots              &       &    &                 &      &\vdots   \\
0      &\ldots& 0    &\scriptscriptstyle{p_{t_0,(t_0+1)}}&\ldots&\scriptscriptstyle{p_{t_0,(t_0+t_1)}}&       &    &                 &      &\scriptscriptstyle{p_{t_0,N}} \\
      &      &      & 0             &\ldots& 0               &       &    &                 &      &\vdots    \\
       &      &      & \vdots        &      & \vdots          &       &    &                 &      &           \\
       &      &      & 0             &\ldots& 0               &       &    &                 &      &        \\
       &      &      &               &      &                 &\ddots &    &\scriptscriptstyle{p_{N-t_l,N-t_l+1}}&\ldots&\scriptscriptstyle{p_{N-t_l,N}}         \\
       &      &      &               &      &                 &       &    & 0               &\ldots&0          \\
       &      &      &               &      &                 &       &    & \vdots          &      &\vdots      \\
       &      &      &               &      &                 &       &    & 0               &\ldots&0          \\
                   };
                   \draw(m-1-1.north west) rectangle (m-3-3.south east-|m-1-3.north east);
                   \draw(m-4-4.north west) rectangle (m-6-6.south east-|m-4-6.north east);
                  \draw(m-8-9.north west) rectangle (m-10-11.south east-|m-8-11.north east);
           \end{tikzpicture},
\end{equation*}

\noindent where the first zero block is a $t_0\times t_0$ matrix, the second zero block is a $t_1\times t_1$ matrix, and so on. If there is a basis in which $\mathbf{D}$ appears as above, the span of the first $t_0$ basis vectors gives us $F^0$, the span of the first
 $(t_0+t_1)$ basis vectors gives us $F^{1}$, and so on, and lastly  the span of all $(t_0+\ldots+t_l)$ basis vectors gives us $F^l=F$.
 So we have a finite increasing filtration of $\mathbf{D}$ by submodules $F^i$ for $i\in\{0,\ldots,l\}$ such that each successive quotient is free.
This allows us to associate the matrix $\mathbf{D}$ of type $t=(t_0,t_1,\ldots,t_l)$ with the flag $F$ of $(l+1)$ folds, where $t_i=\Rank_S(F^i/F^{i-1})$.
\end{remark}

\section{Borel Orbits of Square-Zero Matrices}

 We stratify the variety corresponding to the image of the morphism defined by the matrix $\mathbf{D}$ in Conjecture \ref{Berrinconjecture} by using Borel orbits.
Much of the work in this section can be found in \cite[Section~$3.2-3.4$]{Berrin} and \cite[Chapter~$2$]{berrinPhD}. We recall it here for the reader's convenience.

Let $V_N$ be the variety of square-zero strictly upper triangular $N\times N$ matrices over $k$. We consider the action of the Borel subgroup $B_N$ of $\mathrm{GL_N(k)}$ on $V_N$. The group $B_N$ acts on the variety $V_N$ by conjugation and there is a nice representative for each Borel orbit. A \emph{partial permutation matrix} is a matrix that has at most one entry of $1$ in each row and column and $0$s elsewhere. Rothbach \cite[Theorem~$1$]{Rothbach} verified that each $B_N$-orbit of $V_N$ contains a unique partial permutation matrix. Each of these partial permutation matrices can be identified with an involution $\sigma\in \mathrm{Sym}(N)$. Let $P$ be a partial permutation matrix in the set of non-zero $N\times N$ strictly upper triangular square-zero partial permutation matrices. The one-to-one correspondence between $P$ and $\sigma$ is given by
$$(P)_{ij}=1 \Leftrightarrow \sigma(i)=j \mbox{ and } i<j .$$

\begin{remark} The closure of a Borel orbit is the closure of an image of the Borel group, which is an
irreducible variety; hence, these closures are irreducible varieties themselves, see \cite{Oliver}. The variety $V_N$ is a finite union of closures of all Borel orbits (irreducible varieties), which are partially ordered by inclusion.
\end{remark}
 In order to identify which Borel orbits are contained in the closure of a given Borel orbit, in terms of the corresponding involutions (or matrices), we define a partial order as follows. For two given involutions $\sigma$ and $\sigma'$ corresponding to the non-zero $N\times N$ strictly upper triangular square-zero partial permutation matrices, we have $\sigma'\leq \sigma$ if $\sigma '$ can be
obtained from $\sigma$ by a sequence of moves of the following form:
\begin{itemize}
\item Type {\RN{1}} replaces $\sigma$ with $\sigma (i,j)$ if $\sigma(i)=j$
and $i\neq j$,
\item Type {\RN{2}} replaces $\sigma$ with $\sigma^{(i,i')}$ if
$\sigma(i)=i<
i' < \sigma(i')$,
\item Type {\RN{3}} replaces $\sigma$ with $\sigma^{(j,j')}$ if
$\sigma(j)<\sigma(j')<j'<j$,
\item Type {\RN{4}} replaces $\sigma$ with $\sigma^{(j,j')}$ if $
\sigma(j')< j'< j=\sigma(j)$,
\item Type {\RN{5}} replaces $\sigma$ with $\sigma^{(i,j)}$ if
$i<\sigma(i)<\sigma(j)<j$,
\end{itemize}
where $\sigma (i,j)$ denotes the
product of the permutations $\sigma $ and the transposition $(i,j)$, and $\sigma^{(i,j)}$ denotes the right-conjugate
of $\sigma$ by $(i,j)$.

 Let $d:=(d_1,d_2,\dots ,d_N)$ be an $N$-tuple of nonincreasing integers and $n:=N/2$. Let $V(d,n)$ denote the weighted quasi-projective variety that consists of rank $n$ square-zero upper triangular $N\times N$ matrices $(x_{ij})$ under the equivalence relation
 $({\lambda}^{\mathrm{d_i-d_j+1}}x_{ij})\sim (x_{ij})$ for all $\lambda$ in the unit group $k^{\ast}$. Then the matrix $\mathbf{D}$ represents a nonconstant morphism $\psi$ from the projective variety $\mathbb{P}^{r-1}_k$ to $V(d,n)$. As in  \cite[Section~$3.4$]{Berrin}, there is a lift of this morphism to a morphism from $\mathbb{A}^{r}_k-\{0\}$ to the cone over $V(d,n)$ that can
be extended to a morphism $\widetilde{\psi }:\mathbb{A}^{r}_k\rightarrow V_N$. Note that the order on Borel orbits can be expressed geometrically as one being contained in the closure of the other one, see \cite[Theorem~$4.8$]{Oliver}.
Since  $\mathbb{A}^{r}_k$ is an irreducible affine variety, there exists a unique
Borel orbit that is maximal among the Borel orbits that intersects the
image of $\widetilde{\psi }$ nontrivially. Hence we may associate a permutation $\sigma$ to the nonconstant
morphism $\psi $ or, equivalently, $\mathbf{D}$:
\begin{itemize}
\item[$\ast$] $\sigma _{\psi }$ is the permutation that corresponds to the unique
Borel orbit that is maximal among all Borel orbits that intersect the image of $\psi$ nontrivially.
\end{itemize}
Since every point in the image of a morphism $\psi $ or, equivalently, $\mathbf{D}$ must have rank
$n$, $\sigma _{\psi }$ must be a product of $n$ distinct transpositions.

 Let $V(d,n)_{\RowNumber  \ColNumber  }$ be the subvariety of $V(d,n)$ given by $x_{ij}=0$ for ${i\geq N-\RowNumber  +1}$ or $j\leq \ColNumber  $. Then $V(d,n)_{\RowNumber  \ColNumber  }$ corresponds to the variety containing the image of the morphism $\mathbf{D}$ that satisfies all six conditions of Conjecture \ref{Berrinconjecture}.
Let $\mathbf{RP}(N)$ denote the permutations in the set of involutions in $\mathrm{Sym}(N)$  of rank $n$. We have already shown that the only possible moves are of types {\RN{3}} or {\RN{5}} between
two permutations in $\mathbf{RP}(N)$; see \cite[Section~$3.6$]{Berrin}.

Now, we consider a move of type $\RN{3}$ more closely. Let us represent a
permutation $\sigma=(i_1j_1)(i_2j_2)\ldots(i_nj_n)$ in $\mathbf{RP}(N)$ by
\begin{align*}
\left(\begin{array}{cccc}
  i_1 & i_2 & \ldots & i_n  \\
  j_1 & j_2 & \ldots & j_n
  \end{array}\right).
\end{align*}
Then move of type $\RN{3}$ swaps $j_a$ with $j_b$ if $j_a>j_b$ for $1\leq a< b\leq n$. Clearly, a move of type $\RN{3}$ preserves the rank of $\sigma$. Also observe that for a given $\sigma\in \mathbf{RP}(N)$ if the corresponding partial permutation $P$ is of type $(t_0,\ldots,t_l)$, then moves of type $\RN{3}$ do not change the numbers $t_0=\ColNumber$ and $t_l=\RowNumber$.

 For instance in \cite{Berrin} when $N=6$, we have the following Hasse diagram with four levels $\mathrm{L_1},\ldots \mathrm{L_4}$,  where each dotted line denotes a move of type \RN{3} and solid line
denotes a move of type \RN{5}:
\begin{figure}[H]
    \centering
    \resizebox{0.9\textwidth}{!}{%
\begin{tikzpicture}
\node (L1-0) at (-12,3) {$\mathrm{L}_1$};
  \node (one) at (-8,3) {$\tiny{\left(
\begin{array}{ccc}
  1 & 3 & 5 \\2 & 4 & 6 \end{array}\right)}$};
  \node (two) at (-4,3) {$\tiny{\left(\begin{array}{ccc}
  \boldsymbol{1} & \boldsymbol{2} & \boldsymbol{5} \\ \boldsymbol{4} & \boldsymbol{3} & \boldsymbol{6} \end{array}\right)}$};
  \node (three) at (0,3) {$\tiny{\left(
\begin{array}{ccc}
  1 & 3 & 4 \\ 2 & 6 & 5 \end{array}\right)}$};
  \node (four) at (4,3) {$\tiny{\left(\begin{array}{ccc}
  1 & 2 & 4 \\6 & 3 & 5 \end{array}\right)}$};
  \node (five) at (8,3) {$\tiny{\left(\begin{array}{ccc}
  1 & 2 & 3 \\6 & 5 & 4 \end{array}\right)}$};
  \node (L2-0) at (-12,0) {$\mathrm{L}_2$};
  \node (six) at (-10,0) {$\tiny {\left(
\begin{array}{ccc}
  \boldsymbol{1} & \boldsymbol{2} & \boldsymbol{5} \\ \boldsymbol{3} & \boldsymbol{4} & \boldsymbol{6} \end{array}\right)}$};
  \node (seven) at (-6,0) {$\tiny{\left(\begin{array}{ccc}
  1 & 3 & 4 \\2 & 5 & 6 \end{array}\right)}$};
  \node (eight) at (-2,0) {$\tiny {\left(\begin{array}{ccc}
  1 & 2 & 4 \\5 & 3 & 6 \end{array}\right)}$};
  \node (nine) at (2,0) {$ \tiny{\left(\begin{array}{ccc}
  1 & 2 & 4 \\3 & 6 & 5 \end{array}\right)}$};
  \node (ten) at (6,0) {$ \tiny{\left(\begin{array}{ccc}
  1 & 2& 3 \\5 & 6 & 4 \end{array}\right)}$};
  \node (eleven) at (10,0) {$ \tiny {\left(\begin{array}{ccc}
  1 & 2 & 3 \\6 & 4 & 5 \end{array}\right)}$};
  \node (L3-0) at (-12,-3) {$\mathrm{L}_3$};
  \node (twelve) at (-6,-3) {$\tiny {\left(\begin{array}{ccc}
  1 & 2 & 4 \\3 & 5 & 6 \end{array}\right)}$};
  \node (thirteen) at (0,-3) {$\tiny {\left(\begin{array}{ccc}
  1 & 2 & 3 \\5 & 4 & 6 \end{array}\right)}$};
  \node (fourteen) at (6,-3) {$\tiny {\left(\begin{array}{ccc}
  1 & 2 & 3 \\4 & 6 & 5 \end{array}\right)}$};
    \node (L4-0) at (-12,-6) {$\mathrm{L}_4$};
  \node (fifteen) at (0,-6) {$\tiny {\left(\begin{array}{ccc}
  1 & 2 & 3 \\4 & 5 & 6 \end{array}\right)}$};
  \draw (fifteen) -- (twelve);
  \draw (one) to node {$\begin{array}{cc}
 {\RN{5}} & \\ &  \end{array}$} (six);
  \draw (twelve) -- (seven) -- (one);
  \draw(one) -- (ten);
  \draw[dashed] (ten) -- (fourteen) -- (fifteen);
  \draw[thick, dashed,->](two)  to node {$\begin{array}{cc}
  & {\RN{3}}\\ &  \end{array}$} (six);
  \draw (six) -- (thirteen);
  \draw (six) -- (twelve);
  \draw (two) -- (eight) -- (thirteen);
  \draw[dashed] (three) -- (seven);
  \draw (three) -- (nine) -- (fourteen);
  \draw[dashed] (four) -- (eight) -- (twelve);
  \draw[dashed] (four) -- (nine) -- (twelve);
  \draw[dashed] (five) -- (ten) -- (thirteen) -- (fifteen);
  \draw (four) -- (eleven);
  \draw (seven) -- (fourteen);
  \draw[dashed] (eleven) -- (thirteen);
  \draw[dashed] (five) to node {$\begin{array}{cc} & {\RN{3}}\\ &  \end{array}$}
(eleven) -- (fourteen);
\end{tikzpicture}
 }%
    \caption{Hasse diagram of $\mathbf{RP}(6)$}
    \label{fig:2}
  \end{figure}

When $N=6$, the image of the matrix $\mathbf{D}$ in Conjecture \ref{Berrinconjecture} might be given by any element that appears in $\mathbf{RP}(6)$ in Figure $1$. Let's say
 $\sigma'={\left(
\begin{array}{ccc}
  1 & 2 & 5 \\3 & 4 & 6 \end{array}\right)}$, or equivalently, $\sigma'=(13)(24)(56)$. The involution $\sigma'$ corresponds to the matrix $\mathbf{D}$ in the following form;
\begin{equation*}
   \mathbf{D} = \begin{tikzpicture}[baseline={(m.center)}]
           \matrix [matrix of math nodes,left delimiter={[},right delimiter={]}] (m)
           {          0 & 0     & \mathbf{p_{13}} & p_{14} & p_{15} & p_{16}\\
                      0 & 0     & 0      & \mathbf{p_{24}} & p_{25} & p_{26}\\
                      0 & 0     & 0      & 0      & 0      & p_{36}\\
                      0 & 0     & 0      & 0      & 0      & p_{46}\\
                      0 & 0     & 0      & 0      & 0      & \mathbf{p_{56}} \\
                      0 & 0     & 0      & 0      & 0      & 0\\
                     };
                     \draw(m-1-1.north west) rectangle (m-2-2.south east-|m-1-2.north east);
                   \draw(m-3-3.north west) rectangle (m-5-5.south east-|m-3-5.north east);
                  \draw(m-6-6.north west) rectangle (m-6-6.south east-|m-6-6.north east);
           \end{tikzpicture}.
\end{equation*}
Note that the matrix $\mathbf{D}$ has block type $t=(t_0,t_1,t_2)=(2,3,1)$ with $t_0=\ColNumber$ and $t_2=\RowNumber$.

One can obtain $\sigma'=(13)(24)(56)$ in the level $\mathrm{L}_2$ from the maximal element $\sigma=(14)(23)(56)$ in the level $\mathrm{L}_1$ by a move of type \RN{3}. More precisely, set $(j,j')=(4 3)$, then we have $\sigma(4)<\sigma(3)<3<4$. So one can replace $\sigma$ with   $(\sigma)^{(43)}=(43)(14)(23)(56)(43)^{-1}=(13)(24)(56)$ that means $(13)(24)(56) \leq (14)(23)(56)$. The corresponding matrix for the maximal element $(14)(23)(56)$ is
\begin{equation*}
   \mathbf{D} = \begin{tikzpicture}[baseline={(m.center)}]
           \matrix [matrix of math nodes,left delimiter={[},right delimiter={]}] (m)
           {          0 & 0     & p_{13}           & \mathbf{p_{14}} & p_{15} & p_{16}\\
                      0 & 0     & \mathbf{p_{23}}& p_{24} & p_{25} & p_{26}\\
                      0 & 0     & 0      & 0      & 0      & p_{36}\\
                      0 & 0     & 0      & 0      & 0      & p_{46}\\
                      0 & 0     & 0      & 0      & 0      & \mathbf{p_{56}} \\
                      0 & 0     & 0      & 0      & 0      & 0\\
                     };
                     \draw(m-1-1.north west) rectangle (m-2-2.south east-|m-1-2.north east);
                   \draw(m-3-3.north west) rectangle (m-5-5.south east-|m-3-5.north east);
                  \draw(m-6-6.north west) rectangle (m-6-6.south east-|m-6-6.north east);
           \end{tikzpicture}.
\end{equation*}
Clearly the numbers $ t_0$ and $t_2$ are the same as before.

 In general, we showed that when $N=2,4$, or $6$, there exists a unique maximal
 $\sigma \in \mathbf{RP}(N)$ such that $\sigma' $ can be
obtained from $\sigma $ by a sequence of moves of type \RN{3} in \cite{Berrin}. In the Appendix of this paper we see the same holds for $N=8$, see Figure \ref{fig:1}.
Since moves of type \RN{3} do not change the number of leading zero columns $\ColNumber$ and ending zero rows $\RowNumber$ of the corresponding partial permutation matrices,
the Borel orbit corresponding to $\sigma$ is contained in
$V(d,n)_{\mathcal{R}\mathcal{C}}$
if and only if the Borel orbit corresponding to $\sigma '$ is
 contained in $V(d,n)_{\mathcal{R}\mathcal{C}}$ for all $\mathcal{R}, \mathcal{C}$.
Hence we restrict our
attention to maximal elements in $\mathbf{RP}(N)$  as it is enough to consider
those in the proof of Theorem \ref{MainTheorem}.

\section{The Main Result}
\label{MainSection}
One of the main facts we will use to prove our main result, known as Multivariate Fundamental Theorem of Algebra, is stated as follows:
\begin{theorem}\cite[Theorem~$2.1.3$]{berrinPhD}\label{thmPOLY}
Let $f_1, f_2,\ldots, f_s$ be nonconstant homogeneous polynomials in $k[x_1,\ldots,x_r]$, where $k$ is an algebraically closed field. If $r>s$, then there exists $\gamma\in \mathbb{P}^{r-1}_k$ such that $f_1(\gamma)=0$, $f_2(\gamma)=0, \ldots, f_s(\gamma)=0$.
\end{theorem}

\begin{theorem}\label{MainTheorem}
Let $k$ be an algebraically closed field and $S:=k[x_1,\ldots, x_r]$ the polynomial ring over $k$ in the variables $x_1,\ldots,x_r$ of degree $1$. Assume that $\RowNumber$, $\ColNumber$ are positive integers, and
$\mathbf{D}=( p_{ij})\in \Mat_{8\times 8}(S)$. If
\begin{enumerate}[label=\arabic*)]
\item $\mathbf{D}$ is strictly upper triangular,
\item $\mathbf{D}^2=0$,
\item $0< \dimension_k \Homology(\mathbf{D})< \infty$,
\item for all $i$ and $j$, we have $p_{ij}(0,\ldots,0)=0$,
\item for all $(a_1,\ldots,a_r)\in k^r-\{(0,\ldots,0)\}$, we have $\Rank(\mathbf{D}(a_1,\ldots,a_r))=4$,
\item there exists an $N$-tuple of nonincreasing integers $(d_1,\ldots,d_{N})$ such that for all $i$ and $j$, $p_{ij}$ is a homogeneous polynomial in $S$ of degree $d_i-d_j+1$,
\item when  $j\leq \ColNumber$  or  $i\geq N-\RowNumber +1$, we have $p_{ij}=0$,
\end{enumerate}
then $8\geq 2^{r-1}\left(\RowNumber +\ColNumber \right)$.
\end{theorem}
As noted above, to prove Conjecture \ref{Berrinconjecture} for $N=8$, it suffices to consider the maximal elements
 in $\mathbf{RP}(8)$, which are:
\begin{align*}
&\MAX\mathbf{RP}(8):=\big\{ {(12)(34)(56)(78),\;   (14)(23)(56)(78),\;  (12)(34)(58)(67), \; (14)(23)(58)(67)},\\
& {(12)(36)(45)(78),\;   (12)(38)(47)(56),\; (16)(25)(34)(78), \; (16)(23)(45)(78), \; (12)(38)(45)(67)},\\
& {(18)(23)(47)(56),\;  (18)(25)(34)(67), \; (18)(23)(45)(67),\;  (18)(27)(34)(56), \; (18)(27)(36)(45)} \big\}.
\end{align*}
Note that we can reduce the number of cases from $14$ to $10$ because of symmetry with respect to the anti-diagonal of the corresponding partial permutation matrices. This symmetry takes a matrix of type $(t_0,\ldots,t_l)$ to one of type $(t_l,\ldots,t_0)$. Hence it is enough to check the list
\begin{align*}
&\big\{ (12)(34)(56)(78),\; (12)(34)(58)(67),\; (14)(23)(58)(67),\; (12)(36)(45)(78),\;(12)(38)(47)(56),\\
&      (16)(23)(45)(78),\;  (18)(23)(47)(56),\;  (18)(23)(45)(67),\;  (18)(27)(34)(56),\;  (18)(27)(36)(45)\big\}.
\end{align*}

Remember that the matrix $\mathbf{D}$ in Theorem \ref{MainTheorem} has the block type $t=(t_0,t_1,\ldots,t_l)$, and $N=8=\sum\limits_{i=0}^{l}t_i$ with $t_0=\mathcal{C}$ and $t_l=\mathcal{R}$. We also have some restrictions on $t_i$'s by the work in \cite{Avramov}.
By Theorem \ref{Theoremrandl}, if the type $t$ has the length of $l+1$ and $r$ is the number of variables of the polynomial ring, then $l\geq r$. Indeed, by Lemma \ref{lemmaAvramov} (b), $t_i$ must be at least $2$ for $i\in \{2,3,\ldots, l-1\}$. In other words,  only $t_0$ or $t_l$ might be $1$, so $t_i=1$ cannot appear in the middle. The types corresponding to the involutions, respectively, are as follows:
\begin{align*}
\big\{ (3,2,3),  (2,2,3,1),  (1,2,3,2), (4,4),  (1,2,2,2,1),  (2,2,2,2),  (1,3,3,1),  (1,4,3),  (2,4,2),  (2,3,3) \big\}.
\end{align*}

We will also use the following notation in several cases: For a matrix $X$, the minor
$m^{i_1\,i_2\,\dots \,i_k}_{j_1\, j_2\, \dots \,j_k}(X)$ is the
determinant of
the $k\times k$ submatrix obtained by taking the $i_1\textsuperscript{th}$,
$i_2\textsuperscript{th},\dots ,{i_k}\textsuperscript{th}$ rows and
${j_1}\textsuperscript{th},{j_2}\textsuperscript{th}, \dots,{j_k}\textsuperscript{th}$ columns of $X$. In each case the ring is $\mathrm{UFD}$, so the greatest common divisor ($\mathrm{gcd}$) of any two polynomials $p_1$ and $p_2$ is a polynomial. The greatest common divisor of three polynomials is defined as $\mathrm{gcd}(p_1,p_2,p_3):=\mathrm{gcd}(p_1, \mathrm{gcd}(p_2,p_3))$.

The proof is given by the following ten propositions:
\begin{proposition}\label{prop1}
 If $\sigma \leq (18)(27)(34)(56)$, then $r<2$.
\end{proposition}

\begin{proof}
Suppose to the contrary that $\sigma \leq (18)(27)(34)(56)$ and $r\geq 2$. By Remark \ref{remarkvariable}, it is enough to consider $r=2$. The matrix $\mathbf{D}$ corresponding to the involution ${\sigma \leq (18)(27)(34)(56)}$ is in the form
\begin{equation*}
   \mathbf{D} = \begin{tikzpicture}[baseline={(m.center)}]
           \matrix [matrix of math nodes,left delimiter={[},right delimiter={]}] (m)
           {          0 & 0 & 0 & p_{14} & p_{15} & p_{16} & p_{17} & p_{18} \\
                      0 & 0 & 0 & p_{24} & p_{25} & p_{26} & p_{27} & p_{28} \\
                      0 & 0 & 0 & p_{34} & p_{35} & p_{36} & p_{37} & p_{38} \\
                      0 & 0 & 0 & 0      & 0      & p_{46} & p_{47} & p_{48} \\
                      0 & 0 & 0 & 0      & 0      & p_{56} & p_{57} & p_{58} \\
                      0 & 0 & 0 & 0      & 0      & 0      & 0      & 0      \\
                      0 & 0 & 0 & 0      & 0      & 0      & 0      & 0      \\
                      0 & 0 & 0 & 0      & 0      & 0      & 0      & 0      \\
                     };
                     \draw(m-1-1.north west) rectangle (m-3-3.south east-|m-1-3.north east);
                   \draw(m-4-4.north west) rectangle (m-5-5.south east-|m-4-5.north east);
                  \draw(m-6-6.north west) rectangle (m-8-8.south east-|m-6-8.north east);
           \end{tikzpicture},
\end{equation*}
where the $p_{ij}$ are homogeneous polynomials in $k[x_1,x_2]$ of degree $d_i-d_j+1$. The block type of the matrix $\mathbf{D}$ is $t=(t_0, t_1, t_2)=(3, 2, 3)$ with $\ColNumber=\RowNumber=3$. Then at least one of the three pairs $(p_{14}, p_{15})$, $(p_{24}, p_{25})$, and $(p_{34}, p_{35})$ must be non-zero. Otherwise, the block type of $\mathbf{D}$ turns into $t=(5,3)$ which implies $l=1$. By Theorem \ref{Theoremrandl} we have $l\geq r$, but this fact contradicts the assumption that $r\geq 2$. Similarly, $p_{14}, p_{24}$, and $p_{34}$ cannot vanish together because it is not possible to have the type $(4,1,3)$ by Lemma \ref{lemmaAvramov}(b). Suppose that $p_{14}\neq 0$. Then $p_{15}$ must be non-zero too. Otherwise, since ${\mathbf{D}}^2=0$, we would have $p_{14}p_{46}=0$, $p_{14}p_{47}=0$, and $p_{14}p_{48}=0$. Since $k[x_1,x_2]$ is $\mathrm{UFD}$, the fourth row of $\mathbf{D}$ would be completely zero. Then we could interchange the fourth row with the fifth row, and the fourth column with the fifth column at the same time. If $p_{25}\neq 0$ or $p_{35}\neq 0$, because of this swap, the type of $\mathbf{D}$ would be $(3,1,4)$ which contradicts Lemma \ref{lemmaAvramov}(b). If $p_{25}$ and $p_{35}$ are both zero, then due to the swap, the type of $\mathbf{D}$ would be $(4,4)$ with $l=1$ which contradicts Theorem \ref{Theoremrandl}. The case $p_{15}\neq 0$ and $ p_{14}=0$ can be proved by a contradiction using  a similar argument.
Hence, $p_{14}\neq 0$ if and only if $p_{15}\neq 0$.
Note that if $p_{14}=p_{15}=0$, then when we consider the fourth column, we know that $p_{24}\neq 0$ or $p_{34}\neq 0$. If $p_{24}\neq 0$ and $p_{34}=0$, we can add the second row to the first row. If $p_{24}= 0$ and $p_{34}\neq 0$, we can add the third row to the first row.
If $p_{24}\neq 0$ and $p_{34}\neq 0$, we can add the second and the third row to the first row. Since the first three columns are already zero, the square of the matrix is still zero. These elementary operations do not change the rank of the matrix $\mathbf{D}$. Therefore, we can set $p_{14}$ to be non-zero, so that $p_{15}$ is also non-zero. Using similar arguments, swapping and adding rows and columns, we can set that the other pairs $(p_{24},p_{25})$, and $(p_{34},p_{35})$ are also completely non-zero.

 Moreover, at least one of the pairs $(p_{46},p_{56})$, $(p_{47},p_{57})$ and $(p_{48},p_{58})$ must be non-zero. Otherwise the block type of $\mathbf{D}$ turns into $(3,5)$ with $l=1$ which contradicts Theorem \ref{Theoremrandl}. Since the matrix $\mathbf{D}$ is a square-zero matrix, we have ${p_{14}p_{46}+p_{15}p_{56}=0}$. We already obtained $p_{14}\neq 0$ and $p_{15}\neq 0$ above. Thus, if either $p_{46}$ or $p_{56}$ is non-zero, the other must also be non-zero. Since $\mathbf{D}^2=0$, this is also valid for the pairs $(p_{47}, p_{57})$ and $(p_{48}, p_{58})$.
 Note that if $p_{46}=p_{56}=0$, then when we consider the fourth row, we know that $p_{47}\neq 0$ or $p_{48}\neq 0$. If $p_{47}\neq 0$ and $p_{48}=0$, we can add the seventh column to the sixth column. If $p_{47}= 0$ and $p_{48}\neq 0$, we can add the eighth column to the sixth column. If $p_{47}\neq 0$ and $p_{48}\neq0$, we can add the seventh and the eighth column to the sixth column. Since the last three rows are zero, the square of the matrix remains zero. These elementary operations do not change the rank of the matrix $\mathbf{D}$. Therefore, we can set  $p_{46}$ to be non-zero, so that $p_{56}$ is also non-zero. Using similar arguments, adding rows and columns, we can set that the other pairs $(p_{47},p_{57})$, and $(p_{48},p_{58})$ are also completely non-zero.

Further suppose that we have the following polynomials that are the greatest common divisor of completely non-zero pairs:
\begin{align*}
a:=\mathrm{gcd}(p_{14},p_{15}), & \qquad p_{14}=\overline{p_{14}}a, \qquad p_{15}=\overline{p_{15}}a,\\
b:=\mathrm{gcd}(p_{24},p_{25}), & \qquad p_{24}=\overline{p_{24}}b, \qquad p_{25}=\overline{p_{25}}b,\\
c:=\mathrm{gcd}(p_{34},p_{35}), & \qquad p_{34}=\overline{p_{34}}c, \qquad p_{35}=\overline{p_{35}}c,\\
d:=\mathrm{gcd}(p_{46},p_{56}), & \qquad p_{46}=\overline{p_{46}}d, \qquad p_{56}=\overline{p_{56}}d, \\
e:=\mathrm{gcd}(p_{47},p_{57}), & \qquad p_{47}=\overline{p_{47}}e, \qquad p_{57}=\overline{p_{57}}e, \\
f:=\mathrm{gcd}(p_{48},p_{58}), & \qquad p_{48}=\overline{p_{58}}f, \qquad p_{58}=\overline{p_{58}}f.
\end{align*}
 By the notation above we get $\overline{p_{14}}a\overline{p_{46}}d+\overline{p_{15}}a\overline{p_{56}}d=0$ which implies that ${\overline{p_{14}}\,\overline{p_{46}}=-\overline{p_{15}}\,\overline{p_{56}}}$. Since $k[x_1,x_2]$ is $\mathrm{UFD}$, and  $\overline{p_{14}}$ and $\overline{p_{15}}$ are relatively prime, $\overline{p_{56}}$ must divide $\overline{p_{14}}$. Therefore, we may write $u\overline{p_{14}}=\overline{p_{56}}$ and $-u\overline{p_{15}}=\overline{p_{46}}$, where $u$ is a unit.
 Similarly, for some units $v$, $w$, $s$ and $t$ we have
$$\mathbf{D}=\left[ \begin{array}{cccccccc}
           0 & 0 & 0 &  \overline{p_{14}}a &  \overline{p_{15}}a & p_{16}              & p_{17}              & p_{18} \\
           0 & 0 & 0 & v\overline{p_{14}}b & v\overline{p_{15}}b & p_{26}              & p_{27}              & p_{28} \\
           0 & 0 & 0 & w\overline{p_{14}}c & w\overline{p_{15}}c & p_{36}              & p_{37}              & p_{38} \\
           0 & 0 & 0 & 0                   & 0                   & -u\overline{p_{15}}d & -s\overline{p_{15}}e & -t\overline{p_{15}}f\\
           0 & 0 & 0 & 0                   & 0                   & u\overline{p_{14}}d & s\overline{p_{14}}e & t\overline{p_{14}}f\\
           0 & 0 & 0 & 0                   & 0                   & 0      & 0      & 0      \\
           0 & 0 & 0 & 0                   & 0                   & 0      & 0      & 0      \\
           0 & 0 & 0 & 0                   & 0                   & 0      & 0      & 0      \\
                    \end{array}
 \right].
$$
Define
$$T:=\left[ \begin{array}{cccc}
                      a  & p_{16} & p_{17} & p_{18} \\
                      vb & p_{26} & p_{27} & p_{28} \\
                      wc & p_{36} & p_{37} & p_{38} \\
                       0 & ud     & se     & tf
                    \end{array}
                    \right].
$$
Then $\det T$ is a homogeneous polynomial. By Theorem \ref{thmPOLY}, there exists  $\gamma\in \mathbb{P}^{1}_k$ such that $\det T (\gamma)=0$. Thus, we have
$$\Rank\left[ \begin{array}{ccccc}
    \overline{p_{14}}a(\gamma)&  \overline{p_{15}}a(\gamma)& p_{16}(\gamma)            & p_{17}(\gamma) & p_{18}(\gamma) \\
   v\overline{p_{14}}b(\gamma)& v\overline{p_{15}}b(\gamma)& p_{26}(\gamma)            & p_{27}(\gamma) & p_{28}(\gamma) \\
   w\overline{p_{14}}c(\gamma)& w\overline{p_{15}}c(\gamma)& p_{36}(\gamma)            & p_{37}(\gamma) & p_{38}(\gamma) \\
             0                &  0                         &-u\overline{p_{15}}d(\gamma)&-s\overline{p_{15}}e(\gamma)&-t\overline{p_{15}}f(\gamma)\\
             0                &  0                         &u\overline{p_{14}}d(\gamma)&s\overline{p_{14}}e(\gamma)&t\overline{p_{14}}f(\gamma)
            \end{array}
 \right]\leq 3 ,
$$
that means the rank of the matrix $\mathbf{D}(\gamma)$ is at most $3$, which is a contradiction.

\end{proof}

\begin{proposition}\label{prop2}
 If $\sigma \leq (16)(23)(45)(78)$ or $\sigma \leq (12)(38)(45)(67)$, then $r<3$.
\end{proposition}
\noindent Note that the matrix $\mathbf{D}$ corresponding to $\sigma \leq (16)(23)(45)(78)$ is in the form of a square-zero upper triangular block matrix of type $t=(2,2,3,1)$. The anti-diagonal flip of such a matrix is of type $(1,3,2,2)$, which corresponds to
a ${\sigma}' \leq (12)(38)(45)(67)$.
\begin{proof}
Suppose to the contrary that $r\geq3$. By Remark \ref{remarkvariable}, we can take $r=3$. Consider
$$\mathbf{D}=\left[ \begin{array}{cccccccc}
                      0 & 0     & p_{13} & p_{14} & p_{15} & p_{16} & p_{17} & p_{18} \\
                      0 & 0     & p_{23} & p_{24} & p_{25} & p_{26} & p_{27} & p_{28} \\
                      0 & 0     & 0      & 0      & p_{35} & p_{36} & p_{37} & p_{38} \\
                      0 & 0     & 0      & 0      & p_{45} & p_{46} & p_{47} & p_{48} \\
                      0 & 0     & 0      & 0      & 0      & 0      & 0      & p_{58} \\
                      0 & 0     & 0      & 0      & 0      & 0      & 0      & p_{68}  \\
                      0 & 0     & 0      & 0      & 0      & 0      & 0      & p_{78}  \\
                      0 & 0     & 0      & 0      & 0      & 0      & 0      & 0
                    \end{array}
 \right],
$$
where the $p_{ij}$ are homogeneous polynomials in $k[x_1,x_2,x_3]$. At least one of the pairs $(p_{13}, p_{14})$ and $(p_{23}, p_{24})$ must be non-zero. Otherwise, the block type of $\mathbf{D}$ turns into $t=(4, 3, 1)$ which implies $l=2$. By Theorem \ref{Theoremrandl} we have $l\geq r$, but this fact contradicts the assumption that $r\geq 3$. Also, $p_{13}$ and $p_{23}$ cannot vanish together because it is not possible to have the type $(3,1,3,1)$ because of Lemma \ref{lemmaAvramov}(b). Suppose that $p_{13}$ is non-zero. Then $p_{14}$ must be non-zero too. Otherwise, since ${\mathbf{D}}^2=0$, we would have $p_{13}p_{35}=0, p_{13}p_{36}=0$, and $p_{13}p_{37}=0$. Since $k[x_1,x_2, x_3]$ is $\mathrm{UFD}$, $p_{35}=p_{36}=p_{37}=0$.
$$\mathbf{D}=\left[ \begin{array}{cccccccc}
                      0 & 0     & p_{13} & p_{14} & p_{15} & p_{16} & p_{17} & p_{18} \\
                      0 & 0     & p_{23} & p_{24} & p_{25} & p_{26} & p_{27} & p_{28} \\
                      0 & 0     & 0      & 0      & p_{35} & p_{36} & p_{37} & p_{38} \\
                      0 & 0     & 0      & 0      & p_{45} & p_{46} & p_{47} & p_{48} \\
                      0 & 0     & 0      & 0      & 0      & 0      & 0      & p_{58} \\
                      0 & 0     & 0      & 0      & 0      & 0      & 0      & p_{68}  \\
                      0 & 0     & 0      & 0      & 0      & 0      & 0      & p_{78}  \\
                      0 & 0     & 0      & 0      & 0      & 0      & 0      & 0
                    \end{array}
 \right]\rightarrow \left[ \begin{array}{cccccccc}
                      0 & 0     & p_{13} & 0      & p_{15} & p_{16} & p_{17} & p_{18} \\
                      0 & 0     & p_{23} & p_{24} & p_{25} & p_{26} & p_{27} & p_{28} \\
                      0 & 0     & 0      & 0      & 0      & 0      & 0      & p_{38} \\
                      0 & 0     & 0      & 0      & p_{45} & p_{46} & p_{47} & p_{48} \\
                      0 & 0     & 0      & 0      & 0      & 0      & 0      & p_{58} \\
                      0 & 0     & 0      & 0      & 0      & 0      & 0      & p_{68}  \\
                      0 & 0     & 0      & 0      & 0      & 0      & 0      & p_{78}  \\
                      0 & 0     & 0      & 0      & 0      & 0      & 0      & 0
                    \end{array}
 \right].$$
 Then we could interchange the third row with the fourth row, and the third column with the fourth column. Due to this swap, if $p_{24}\neq 0$, then the type would be $(2,1,4,1)$ which contradicts Lemma \ref{lemmaAvramov}(b). If $p_{24}=0$, then the type would be $(3,4,1)$ with $l=2$ which contradicts Theorem \ref{Theoremrandl}. The case that $p_{14}\neq 0$ implies $p_{13}\neq 0$ can be proven using a similar argument.

 Note that if $p_{13}=p_{14}=0$, then $p_{23}\neq 0$ or $p_{24}\neq 0$. Then we can add the second row to the first row. Since the first two columns are already zero, the square of the matrix is still zero. This elementary operation do not change the rank of the matrix $\mathbf{D}$. Therefore, we can set $p_{13}$ to be non-zero so that $p_{14}$ is also non-zero.

  Actually, if $p_{23}\neq 0$, then $p_{24}\neq 0$. Otherwise, since ${\mathbf{D}}^2=0$, we would have $p_{23}p_{35}=0, p_{23}p_{36}=0$, and $p_{23}p_{37}=0$.  Since $k[x_1,x_2,x_3]$ is $\mathrm{UFD}$, $p_{35}=p_{36}=p_{37}=0$. Then we could interchange the third row with the fourth row, and the third column with the fourth column at the same time. Then the type of $\mathbf{D}$ would be $(2,1,4,1)$ which contradicts Lemma \ref{lemmaAvramov}(b). Similarly, if $p_{24}\neq 0$, then $p_{23}\neq 0$. Thus, we have $p_{23}\neq 0$ if and only if $p_{24}\neq 0$. Note that if $p_{23}= p_{24}=0$, then we can add the first row to the second row without changing the rank of $\mathbf{D}$ and the fact that ${\mathbf{D}}^2=0$.

   We have $p_{13}\neq 0$, $p_{14}\neq 0$, $p_{23}\neq 0$ and $p_{24}\neq 0$. In addition, all of the pairs $(p_{35}, p_{45})$, $(p_{36}, p_{46})$ and $(p_{37}, p_{47})$ cannot vanish because the type cannot be $(2,5,1)$ due to Theorem \ref{Theoremrandl}. Suppose that $p_{35}\neq 0$. Since ${\mathbf{D}}^2=0$, we have $p_{45}\neq 0$, and vice versa. This is also true for the pairs $(p_{36}, p_{46})$ and $(p_{37}, p_{47})$. Note that if $p_{35}= p_{45}=0$, then when we consider the third row, we know that $p_{36}\neq 0$ or $p_{37}\neq 0$. If $p_{36}\neq 0$ and $p_{37}=0$, then we can add the sixth column to the fifth column and subtract the fifth row from the sixth row.
   If $p_{36}=0$ and $p_{37}\neq 0$, then we can add the seventh column to the fifth column and subtract the fifth row from the seventh row.
   If $p_{36}\neq 0$ and $p_{37}\neq 0$, then we can add the sixth and the seventh column to the fifth column and subtract the fifth row from the sixth row and the seventh row. Since these operations do not change the rank of the matrix $\mathbf{D}$ and preserve $\mathbf{D}^2=0$, we can set $p_{35}\neq 0$ so that $p_{45}\neq 0$. Using similar arguments, we can set that the other pairs $(p_{36}, p_{46})$ and $(p_{37}, p_{47})$ are also completely non-zero.
Suppose that
\begin{align*}
a:= \mathrm{gcd}(p_{13},p_{14}), & \qquad p_{13}=\overline{p_{13}}a, \qquad p_{14}=\overline{p_{14}}a,\\
b:= \mathrm{gcd}(p_{23},p_{24}), & \qquad p_{23}=\overline{p_{23}}b, \qquad p_{24}=\overline{p_{24}}b,\\
c:= \mathrm{gcd}(p_{35},p_{45}), & \qquad p_{35}=\overline{p_{35}}c, \qquad p_{45}=\overline{p_{45}}c,\\
d:= \mathrm{gcd}(p_{36},p_{46}), & \qquad p_{36}=\overline{p_{36}}d, \qquad p_{46}=\overline{p_{46}}d, \\
\textrm{and }e:= \mathrm{gcd}(p_{37},p_{47}), & \qquad p_{37}=\overline{p_{37}}e, \qquad p_{47}=\overline{p_{47}}e.
\end{align*}
Since $\mathbf{D}^2=0$, we have $p_{13}p_{35}+p_{14}p_{45}=0$, which implies that
$\overline{p_{13}}\,\overline{p_{35}}=-\overline{p_{14}}\,\overline{p_{45}}$. Since $k[x_1,x_2,x_3]$ is $\mathrm{UFD}$ and  $\overline{p_{13}}$ and $\overline{p_{14}}$ are relatively prime, $\overline{p_{45}}$ must divide $\overline{p_{13}}$. Therefore, we may write
$u\overline{p_{13}}=\overline{p_{45}}$ and $-u\overline{p_{14}}=\overline{p_{35}}$ where $u$ is a unit.
 Similarly, for some units $v$, $w$ and $s$  we have
$$\mathbf{D}=\left[ \begin{array}{cccccccc}
0 & 0 &  \overline{p_{13}}a &  \overline{p_{14}}a & p_{15} & p_{16}   & p_{17}              & p_{18} \\
0 & 0 & v\overline{p_{13}}b & v\overline{p_{14}}b & p_{25} & p_{26}   & p_{27}              & p_{28} \\
0 & 0 & 0 & 0 &-uc\overline{p_{14}}& -wd\overline{p_{14}} &  -se\overline{p_{14}}              & p_{38}\\
0 & 0 & 0 & 0 &uc\overline{p_{13}}& wd\overline{p_{13}} &  se\overline{p_{13}}              & p_{48}\\
0 & 0 & 0 & 0                   & 0                   & 0      & 0      & p_{58}\\
0 & 0 & 0 & 0                   & 0                   & 0      & 0      & p_{68}      \\
0 & 0 & 0 & 0                   & 0                   & 0      & 0      & p_{78}      \\
0 & 0 & 0 & 0                   & 0                   & 0      & 0      & 0      \\
                    \end{array}
 \right].
$$
Let $\gamma\in \mathbb{P}^{2}_k$ be a root of $\overline{p_{13}}$ and $\overline{p_{14}}$, then the rank of the matrix $\mathbf{D}(\gamma)$ is at most $3$, which is a contradiction.
\end{proof}
\begin{proposition}\label{prop3}
If $\sigma \leq (12)(34)(58)(67)$ or $\sigma \leq (14)(23)(56)(78)$, then $r<3$.
\end{proposition}
\noindent These cases are symmetric, so it is enough to prove the case $\sigma \leq (12)(34)(58)(67)$.
\begin{proof}
Suppose to the contrary that $r\geq3$. By Remark \ref{remarkvariable}, we can take $r=3$. Consider
$$\mathbf{D}=\left[ \begin{array}{cccccccc}
                      0 & p_{12} & p_{13} & p_{14} & p_{15} & p_{16} & p_{17} & p_{18} \\
                      0 & 0      & 0      & p_{24} & p_{25} & p_{26} & p_{27} & p_{28} \\
                      0 & 0      & 0      & p_{34} & p_{35} & p_{36} & p_{37} & p_{38} \\
                      0 & 0      & 0      & 0      & 0      & 0      & p_{47} & p_{48} \\
                      0 & 0      & 0      & 0      & 0      & 0      & p_{57} & p_{58} \\
                      0 & 0      & 0      & 0      & 0      & 0      & p_{67} & p_{68} \\
                      0 & 0      & 0      & 0      & 0      & 0      & 0      & 0      \\
                      0 & 0      & 0      & 0      & 0      & 0      & 0      & 0      \\
                    \end{array}
 \right].
$$

 The block type of the matrix $\mathbf{D}$ is $t=(1,2,3,2)$. Note that $p_{12}$ and $p_{13}$ cannot vanish together. Otherwise, the block type of $\mathbf{D}$ turns into $t=(3, 3, 2)$ which implies $l=2$. By Theorem \ref{Theoremrandl} we know that $l\geq r$, so we have a contradiction. If $p_{12}\neq 0$, then $p_{13}$ must be non-zero too. Otherwise, $\mathbf{D}^2=0$ implies that $p_{24}=p_{25}=p_{26}=0$. We could interchange the second row with the third row, and the second column with the third column. Then the type would be $(2,4,2)$ which contradicts Theorem \ref{Theoremrandl}. Clearly, if $p_{13}\neq 0$, then $p_{12}\neq 0$. Therefore we have $p_{12}\neq 0$ and $p_{13}\neq 0$. At least one of the pairs $(p_{24},p_{34})$, $(p_{25},p_{35})$ and $(p_{26},p_{36})$ must be non-zero. Otherwise, the type of $\mathbf{D}$ would be $(1,5,2)$ which contradicts Theorem \ref{Theoremrandl}. Suppose that $p_{24}\neq 0$. Then  $p_{12}p_{24}+p_{13}p_{34}=0$ implies that $p_{34}\neq 0$. Similarly, if $p_{34}\neq 0$, then $p_{24}\neq 0$. Thus, $p_{24}\neq 0$ if and only if $p_{34}\neq 0$. Actually, one could make a swap so that the pairs $(p_{25},p_{35})$ and $(p_{26},p_{36})$ play the same role of the non-zero pair $(p_{24},p_{34})$ from the argument before. If $p_{24}=p_{34}=0$, then when we consider the second row, we know that $p_{25}\neq 0$ or $p_{26}\neq 0$. If $p_{25}\neq 0$ and $p_{26}=0$, we can add the fifth column to the fourth column and subtract the fourth row from the fifth row. If $p_{25}= 0$ and $p_{26}\neq 0$, we can add the sixth column to the fourth column and subtract the fourth row from the sixth row. If $p_{25}\neq 0$ and $p_{26}\neq 0$, we can add the fifth and sixth column to the fourth column and subtract the fourth row from the fifth and the sixth row. Since these operations do not change the rank of the matrix $\mathbf{D}$ and preserve $\mathbf{D}^2=0$, we can set $p_{24}\neq 0$, so that $p_{34}\neq 0$. Using similar arguments, we can set that the other pairs $(p_{25}, p_{35})$ and $(p_{26}, p_{36})$ are also completely non-zero.


Suppose that
\begin{align*}
a:= \mathrm{gcd}(p_{12},p_{13}), & \qquad p_{12}=\overline{p_{12}}a, \qquad p_{13}=\overline{p_{13}}a,\\
b:= \mathrm{gcd}(p_{24},p_{34}), & \qquad p_{24}=\overline{p_{24}}b, \qquad p_{34}=\overline{p_{34}}b,\\
c:= \mathrm{gcd}(p_{25},p_{35}), & \qquad p_{25}=\overline{p_{25}}c, \qquad p_{35}=\overline{p_{35}}c,\\
\textrm{and }d:= \mathrm{gcd}(p_{26},p_{36}), & \qquad p_{26}=\overline{p_{26}}d, \qquad p_{36}=\overline{p_{36}}d.
\end{align*}
 For some units $u$, $v$ and $w$, we have
$$\mathbf{D}=\left[ \begin{array}{cccccccc}
 0 &\overline{p_{12}}a& \overline{p_{13}}a & p_{14} & p_{15} & p_{16} & p_{17} & p_{18} \\
 0 &0                 & 0      & -u\overline{p_{13}}b & -v\overline{p_{13}}c &-w\overline{p_{13}}d & p_{27} & p_{28} \\
 0 &0                 & 0      & u\overline{p_{12}}b  & v\overline{p_{12}}c  & w\overline{p_{12}}d & p_{37} & p_{38} \\
 0 &0                 & 0      & 0      & 0      & 0      & p_{47} & p_{48} \\
 0 &0                 & 0      & 0      & 0      & 0      & p_{57} & p_{58} \\
 0 &0                 & 0      & 0      & 0      & 0      & p_{67} & p_{68} \\
 0 &0                 & 0      & 0      & 0      & 0      & 0      & 0      \\
 0 &0                 & 0      & 0      & 0      & 0      & 0      & 0      \\
                    \end{array}
 \right].
$$
By Theorem \ref{thmPOLY} there exists  $\gamma \in  \mathbb{P}^{2}_k$ such that $\overline{p_{12}}(\gamma)=0$ and $\overline{p_{13}}(\gamma)=0$. Then the rank of matrix $\mathbf{D}(\gamma)$ can be at most $3$. This is a contradiction.
\end{proof}
\begin{proposition}\label{prop4}
 If $\sigma \leq (18)(27)(36)(45)$, then $r<2$.
\end{proposition}
\begin{proof} Suppose not. By Remark \ref{remarkvariable}, we can take $r=2$. Consider
$$\mathbf{D}=\left[ \begin{array}{cccccccc}
                      0 & 0 & 0 & 0 & p_{15} & p_{16} & p_{17} & p_{18} \\
                      0 & 0 & 0 & 0 & p_{25} & p_{26} & p_{27} & p_{28} \\
                      0 & 0 & 0 & 0 & p_{35} & p_{36} & p_{37} & p_{38} \\
                      0 & 0 & 0 & 0 & p_{45} & p_{46} & p_{47} & p_{48} \\
                      0 & 0 & 0 & 0 & 0      & 0      & 0      & 0 \\
                      0 & 0 & 0 & 0 & 0      & 0      & 0      & 0      \\
                      0 & 0 & 0 & 0 & 0      & 0      & 0      & 0      \\
                      0 & 0 & 0 & 0 & 0      & 0      & 0      & 0      \\
                    \end{array}
 \right].
$$
It is enough to consider a root of the minor $m^{1234}_{5678}(\mathbf{D})$ to prove this case.
\end{proof}
\begin{proposition}\label{prop5}
If $\sigma \leq (12)(34)(56)(78)$, then $r<4$.
\end{proposition}
\begin{proof}
Suppose to the contrary that $r\geq4$. By Remark \ref{remarkvariable}, we may take $r=4$. Then we have
$$\mathbf{D}=\left[ \begin{array}{cccccccc}
                      0 & p_{12} & p_{13} & p_{14} & p_{15} & p_{16} & p_{17} & p_{18} \\
                      0 & 0      & 0      & p_{24} & p_{25} & p_{26} & p_{27} & p_{28} \\
                      0 & 0      & 0      & p_{34} & p_{35} & p_{36} & p_{37} & p_{38} \\
                      0 & 0      & 0      & 0      & 0      & p_{46} & p_{47} & p_{48} \\
                      0 & 0      & 0      & 0      & 0      & p_{56} & p_{57} & p_{58} \\
                      0 & 0      & 0      & 0      & 0      & 0      & 0      & p_{68} \\
                      0 & 0      & 0      & 0      & 0      & 0      & 0      & p_{78} \\
                      0 & 0      & 0      & 0      & 0      & 0      & 0      & 0      \\
                    \end{array}
 \right],
$$
where $p_{ij}$'s are homogeneous polynomials in $k[x_1,x_2,x_3,x_4]$.

The block type of the matrix $\mathbf{D}$ is $(1,2,2,2,1)$ with $l=4$. Then $p_{12}$ and $p_{13}$ cannot vanish together.
 Otherwise, the block type $\mathbf{D}$ would be $(3,2,2,1)$  with $l=3$ which contradicts Theorem \ref{Theoremrandl}.  If $p_{12}\neq 0$, then $p_{13}\neq 0$. Otherwise, the block type $\mathbf{D}$ turns into $(2,1,2,2,1)$ which contradicts Lemma \ref{lemmaAvramov}(b). If $p_{13}\neq 0$, then  $p_{12}\neq0$. Otherwise $\mathbf{D}^2=0$ implies that $p_{34}=p_{35}=0$. Then the type would be $(2,3,2,1)$ which contradicts Theorem \ref{Theoremrandl}. Thus, $p_{12}\neq 0$ and $p_{13}\neq0$.

  Also, at least one of the polynomials $p_{24}$ and $p_{34}$ must be non-zero because the type cannot be $(1,3,1,2,1)$ due to Lemma \ref{lemmaAvramov}(b). Suppose that $p_{24}\neq 0$. Then $p_{25}\neq 0$; otherwise $\mathbf{D}^2=0$ would imply that $p_{24}p_{46}=0$ and $p_{24}p_{47}=0$. Since $k[x_1,x_2, x_3, x_4]$ is $\mathrm{UFD}$, $p_{46}=p_{47}=0$. We could interchange the fourth row with the fifth row, and the fourth column with the fifth column. If $p_{35}\neq 0$, because of this swap the type would be $(1,2,1,3,1)$ which contradicts Lemma \ref{lemmaAvramov}(b).  If $p_{35}=0$, then the type would be $(1,3,3,1)$ with $l=3$ which contradicts Theorem \ref{Theoremrandl}. Note that if $p_{24}=p_{25}=0$, then by replacing the second row with the third row and the second column with the third column, the type becomes $(1,1,3,2,1)$ which contradicts Lemma \ref{lemmaAvramov}(b). Thus, we can assume that both $p_{24}\neq 0$ and $p_{25}\neq 0$. Similarly, we can prove that the polynomials $p_{34}$ and $ p_{35}$ are non-zero.

Now consider the pair $(p_{46},p_{56})$, it cannot be completely zero because the type cannot be $(1,2,3,1,1)$. Suppose that $p_{46}\neq 0$. Since $\mathbf{D}^2=0$, $p_{56}\neq 0$, and vice versa. Also, the pair $(p_{47}, p_{57})$ cannot be completely zero. Otherwise, we could interchange the sixth row with the seventh row and the sixth column with the seventh column, and the type would become $(1,2,3,1,1)$ which contradicts Lemma \ref{lemmaAvramov}(b). Since $\mathbf{D}^2=0$, if $p_{47}\neq 0$, then $p_{57}\neq 0$, and vice versa.

Suppose that
\begin{align*}
a:= \mathrm{gcd}(p_{24},p_{25}), & \qquad p_{24}=\overline{p_{24}}a, \qquad p_{25}=\overline{p_{25}}a,\\
b:= \mathrm{gcd}(p_{34},p_{35}), & \qquad p_{34}=\overline{p_{34}}b, \qquad p_{35}=\overline{p_{35}}b,\\
c:= \mathrm{gcd}(p_{46},p_{56}), & \qquad p_{46}=\overline{p_{46}}c, \qquad p_{56}=\overline{p_{56}}c,\\
\textrm{and }d:= \mathrm{gcd}(p_{47},p_{57}), & \qquad p_{47}=\overline{p_{47}}d, \qquad p_{57}=\overline{p_{57}}d.
\end{align*}
Since $\mathbf{D}^2=0$, we have $p_{24}p_{46}+p_{25}p_{56}=0$, which implies that $\overline{p_{56}}=u\overline{p_{24}}$ and $\overline{p_{46}}=-u\overline{p_{25}}$ for some unit $u$.
Similarly, for some units $v$ and $w$, we have
\begin{equation*}
    \mathbf{D} = \begin{tikzpicture}[baseline={(m.center)}]
            \matrix [matrix of math nodes,left delimiter={[},right delimiter={]}] (m)
            {
 0 & p_{12} & p_{13} & p_{14} & p_{15} & p_{16} & p_{17} & p_{18} \\
 0 & 0      & 0      & a\overline{p_{24}}  & a\overline{p_{25}}   & p_{26} & p_{27}\quad  & p_{28} \\
 0 & 0      & 0      & vb\overline{p_{24}} & vb \overline{p_{25}} & p_{36} & p_{37} \quad   & p_{38} \\
 0 & 0      & 0      & 0      & 0      & -uc\overline{p_{25}} & {-wd\overline{p_{25}}}\quad & p_{48} \\
 0 & 0      & 0      & 0      & 0      &  uc\overline{p_{24}} & {wd\overline{p_{24}}} \quad & p_{58} \\
 0 & 0      & 0      & 0      & 0      & 0      & 0      & p_{68} \\
 0 & 0      & 0      & 0      & 0      & 0      & 0      & p_{78} \\
 0 & 0      & 0      & 0      & 0      & 0      & 0      & 0      \\
                  } ;
            \draw (m-2-4.north west) rectangle (m-5-7.south east-|m-2-7.north east);
            \end{tikzpicture}.
\end{equation*}
Let $T$ be the boxed submatrix of $\mathbf{D}$ shown above. Since $r=4$, by Theorem \ref{thmPOLY} there exists $\gamma \in  \mathbb{P}^{3}_k$ such that $\overline{p_{24}}(\gamma)=\overline{p_{25}}(\gamma)=0$ and $m^{23}_{67}(\mathbf{D})(\gamma)=0$. Hence all $2\times 2$ minors of $T$ at $\gamma$ are zero. This implies that the rank of $T(\gamma)$ is at most $1$, so the total rank of $\mathbf{D}$ at $\gamma$ is at most $3$. This is a contradiction.

\end{proof}
\begin{remark}\label{remarkProp5}
Proposition \ref{prop5} itself also implies Corollary \ref{Corollary1} and implies Corollary \ref{Corollary2}.
\end{remark}
\begin{proof}Let $k=\overline{\mathbb{F}}_2$ and $r\geq 4$. Suppose that $A=k[y_1,\ldots,y_r]$  and $M$ is a free,
 finitely generated $\mathrm{dg}$-$A$-module such that $\Rank_A M=N$ with $0< \dim_k \Homology_*(M)<\infty$. Without loss of generality, we may assume that
 $N$ is the smallest rank of all such $\mathrm{dg}$-$A$-modules. By Lemma \cite[Lemma~$3.2.1$]{berrinPhD}, given such a $\mathrm{dg}$-$A$-module $M$ there exists a minimal module. By abuse of notation $M$ is a minimal module. Further assume that $M$ has a free flag $F$ with  $\Freeclass_A F= l$. By Theorem \ref{Theoremrandl}, we have $l\geq 4$.
By Lemma \ref{lemmaAvramov}, $N\neq 2,4,6$. We can also eliminate the case $N=8$ easily. The only possibility for the sequence of integers $t=(t_0,\ldots,t_l)$ corresponding to $F$ is $(1,2,2,2,1)$. On the other hand,  by Proposition \ref{prop5} we see that if there exists a  flag admits this sequence, then $r<4$. That contradicts our assumption on $r$, so $N$ must be at least $10$.
\end{proof}
Note that Remark \ref{remarkProp5} gives an alternative proof of the result given  by Refai in \cite[Theorem~$4.2$]{Refai1}.
Consequently,  Proposition \ref{prop5} itself verifies the Rank Conjecture when $k=\overline{\mathbb{F}}_2$ and $m=3$.

\begin{proposition}\label{prop6}
If $\sigma \leq (18)(23)(45)(67)$, then $r<3$.
\end{proposition}
\begin{proof}
Suppose not; then $r\geq3$, and by Remark \ref{remarkvariable}, we may take $r=3$. We have
$$\mathbf{D}=\left[ \begin{array}{cccccccc}
                      0 & 0      & p_{13} & p_{14} & p_{15} & p_{16} & p_{17} & p_{18} \\
                      0 & 0      & p_{23} & p_{24} & p_{25} & p_{26} & p_{27} & p_{28} \\
                      0 & 0      & 0      & 0      & p_{35} & p_{36} & p_{37} & p_{38} \\
                      0 & 0      & 0      & 0      & p_{45} & p_{46} & p_{47} & p_{48} \\
                      0 & 0      & 0      & 0      & 0      & 0      & p_{57} & p_{58} \\
                      0 & 0      & 0      & 0      & 0      & 0      & p_{67} & p_{68} \\
                      0 & 0      & 0      & 0      & 0      & 0      & 0      & 0     \\
                      0 & 0      & 0      & 0      & 0      & 0      & 0      & 0      \\
                    \end{array}
 \right].
$$

 The matrix $\mathbf{D}$ has block type $t=(2,2,2,2)$. At least one of the pairs $(p_{13}, p_{14})$ and $(p_{23}, p_{24})$ must be non-zero. Otherwise, the block type of $\mathbf{D}$ would be $t=(4,2,2)$ with $l=2$ which contradicts Theorem \ref{Theoremrandl}.  The pair $(p_{13},p_{23})$ cannot vanish because the type turns into $(3,1,2,2)$ which contradicts Lemma \ref{lemmaAvramov}(b). Suppose that $p_{13}\neq 0$, then $p_{14}\neq 0$. Otherwise, we would have $p_{13}p_{35}=0$, $p_{13}p_{36}=0$. Since $k[x_1,x_2, x_3]$ is $\mathrm{UFD}$, $p_{35}=p_{36}=0$. Then we could interchange the third row with the fourth row, and the third column with the fourth column. If $p_{24} \neq 0$, due to this swap, the type of $\mathbf{D}$ would be $(2,1,3,2)$ which contradicts Lemma \ref{lemmaAvramov}(b). If $p_{24}=0$, due to this swap, the type of $\mathbf{D}$ would be $(3,3,2)$ with $l=2$ which contradicts Theorem \ref{Theoremrandl}. Similarly, if we suppose that $p_{14}\neq 0$, then $p_{13}\neq 0$. Thus, we have $p_{13}\neq 0$ if and only if $p_{14}\neq 0$. Moreover, if $p_{13}=0$ and $p_{14}=0$, then we can add the second row to the first row. Since the first two columns are already zero, the square of the new matrix is still zero and its rank remains the same. Therefore, $p_{13}$ and $p_{14}$  can always be set to be both non-zero. Actually, $(p_{23}, p_{24})$ plays the same role as the pair $(p_{13}, p_{14})$ in the previous argument. Thus, we have $p_{23}\neq 0$ and $p_{24}\neq 0$.

At least one of the pairs $(p_{35}, p_{45})$ and $(p_{36}, p_{46})$ must be non-zero. Otherwise, the block type of $\mathbf{D}$ would be $t=(2,4,2)$ with $l=2$ which contradicts Theorem \ref{Theoremrandl}. Note that $p_{35}$ and $p_{45}$ cannot vanish at the same time. Otherwise the type would be $(2,3,1,2)$ which contradicts Lemma \ref{lemmaAvramov}(b). If $p_{35} \neq 0$, then $p_{45}\neq 0$. Otherwise, we would have $p_{57}=p_{58}=0$. Then we could interchange the fifth row with the sixth row, and the fifth column with the sixth column at the same time. Then, if the pair $(p_{36}, p_{46})$ is non-zero, the type turns into $(2,2,1,3)$ which contradicts Lemma \ref{lemmaAvramov}(b). If the pair $(p_{36}, p_{46})$ is zero, the type turns into $(2,3,3)$ which contradicts Theorem \ref{Theoremrandl}. Similarly,  one can show that $p_{45}\neq 0$, then $p_{35} \neq 0$. Thus, $p_{35} \neq 0$ and $p_{45}\neq 0$. Since $(p_{36}, p_{46})$ plays the same role as the pair $(p_{35}, p_{45})$, we have $p_{36}\neq 0$ and $p_{46}\neq 0$. At least one of the pairs $(p_{57}, p_{67})$ and $(p_{58}, p_{68})$ must be non-zero. Otherwise the block type of $\mathbf{D}$ would be $t=(2,2,4)$ with $l=2$ which contradicts Theorem \ref{Theoremrandl}. Again, if $p_{57}\neq 0$, then $p_{67}\neq 0$, and vice versa because of the fact that ${\mathbf{D}}^2=0$. Thus, $p_{57}\neq 0$ if and only if $p_{67}\neq 0$. Note that if $p_{57}=p_{67}=0$, then we can add the eighth column to the seventh column. Since the last two rows are zero, the square of the new matrix is still zero and its rank remains the same. So, we can set  $p_{57}\neq 0$ and $p_{67}\neq 0$. Similarly $p_{58}\neq 0$ and $p_{68}\neq 0$. Further suppose that:
\begin{align*}
a:= \mathrm{gcd}(p_{13},p_{14}), & \qquad p_{13}=\overline{p_{13}}a, \qquad p_{14}=\overline{p_{14}}a,\\
b:= \mathrm{gcd}(p_{23},p_{24}), & \qquad p_{23}=\overline{p_{23}}b, \qquad p_{24}=\overline{p_{24}}b,\\
c:= \mathrm{gcd}(p_{35},p_{45}), & \qquad p_{35}=\overline{p_{35}}c, \qquad p_{45}=\overline{p_{45}}c,\\
d:= \mathrm{gcd}(p_{36},p_{46}), & \qquad p_{36}=\overline{p_{36}}d, \qquad p_{46}=\overline{p_{46}}d,\\
e:= \mathrm{gcd}(p_{57},p_{67}), & \qquad p_{57}=\overline{p_{57}}e, \qquad p_{67}=\overline{p_{67}}e,\\
\textrm{and } f:= \mathrm{gcd}(p_{58},p_{68}), & \qquad p_{58}=\overline{p_{58}}f, \qquad p_{68}=\overline{p_{68}}f.
\end{align*}
Since $\mathbf{D}^2=0$, we have $p_{13}p_{35}+p_{14}p_{45}=0$, which implies that
$\overline{p_{13}}\,\overline{p_{35}}=-\overline{p_{14}}\,\overline{p_{45}}$. We may write
$u\overline{p_{13}}=\overline{p_{45}}$ and $-u\overline{p_{14}}=\overline{p_{35}}$ where $u$ is a unit. Similarly, $v\overline{p_{13}}=\overline{p_{46}}$ and $-v\overline{p_{14}}=\overline{p_{36}}$ for some unit $v$.  An argument similar to the one for $p_{13},p_{14}$ is used for $p_{23}, p_{24}$.
 To simplify our notation, let $\overline{p_{13}}=g_1$ and $\overline{p_{14}}=g_2$, so that $g_1$ and $g_2$ are relatively prime homogeneous polynomials. Moreover, we keep using $b$ instead of $vb$, $c$ instead of $uc$ etc.
 Since ${\mathbf{D}}^2=0$, ${c}g_2 {e}\overline{p_{57}}+{d}g_2{e}\overline{ p_{67}}=0$ and
 ${c}g_1 {e} \overline{p_{57}}+{d}g_1 {e} \overline{p_{67}}=0$.
Then $\overline{c}\, \overline{p_{57}}=-\overline{d}\, \overline{p_{67}}$, so $\overline{p_{57}}=-\overline{d}w $ and $\overline{p_{67}}=\overline{c}w$, where $w$ is a unit. An argument similar to the one for $p_{57}$, $p_{67}$ is used for $p_{58}$, $p_{68}$ for some unit $w'$. We keep using $e$ instead of $ew$ and $f$ instead of $fw'$. Let $\mathbf{T}$ be the boxed submatrix of $\mathbf{D}$ shown below:
\begin{equation*}
    \mathbf{D} = \begin{tikzpicture}[baseline={(m.center)}]
            \matrix [matrix of math nodes,left delimiter={[},right delimiter={]}] (m)
            {
0 & 0 & a g_1 &  a g_2  & A              & B        & p_{17}      & p_{18} \\
0 & 0 & b g_1 &  b g_2  & C              & D        & p_{27}      & p_{28} \\
0 & 0 & 0     & 0       &\quad {-c g_2}  & -d g_2   & E           & F\\
0 & 0 & 0     & 0       &\quad  {c g_1}  & d g_1    & G           & H\\
0 & 0 & 0     & 0       & 0              & 0        & -e d        & -f d\\
0 & 0 & 0     & 0       & 0              & 0        & e  c        &  f c\\
0 & 0 & 0     & 0       & 0              & 0        & 0           & 0     \\
0 & 0 & 0     & 0       & 0              & 0        & 0           & 0     \\
            };
            \draw (m-1-5.north west) rectangle (m-4-8.south east-|m-1-8.north east);
            \end{tikzpicture}, \mbox{ and define }
\mathbf{M} := \begin{tikzpicture} [baseline={(m.center)}]
            \matrix [matrix of math nodes,left delimiter={[},right delimiter={]}] (m)
            {
 a  & p_{17}   & p_{18} \\
 b  & p_{27}   & p_{28} \\
 0  & e        & f     \\
            };
            \end{tikzpicture}.
\end{equation*}

 By Theorem \ref{thmPOLY} there exists $\gamma\in \mathbb{P}^{2}_k$ such that $\mathrm{det}(\mathbf{T}(\gamma))=0$ and $\mathrm{det}(\mathbf{M}(\gamma))=0$. Then, we have
$$\mathbf{D}(\gamma)=\left[ \begin{array}{cccccccc}
0 & 0 & (a g_1)(\gamma) & (a g_2)(\gamma)  & A(\gamma)        & B(\gamma)         & p_{17}(\gamma)   & p_{18}(\gamma) \\
0 & 0 & (b g_1)(\gamma) & (b g_2)(\gamma)  & C(\gamma)        & D(\gamma)         & p_{27}(\gamma)   & p_{28}(\gamma) \\
0 & 0 & 0               & 0                & -(c g_2)(\gamma) & -(d g_2)(\gamma)  & E(\gamma)        & F(\gamma)\\
0 & 0 & 0               & 0                & (c g_1)(\gamma)  & (d g_1)(\gamma)   & G(\gamma)        & H(\gamma)\\
0 & 0 & 0               & 0                & 0                & 0                 & -(e d)(\gamma)   & -(f d)(\gamma)\\
0 & 0 & 0               & 0                & 0                & 0                 & (e c)(\gamma)    & (f c)(\gamma)\\
0 & 0 & 0               & 0                & 0                & 0                 & 0                & 0     \\
0 & 0 & 0               & 0                & 0                & 0                 & 0                & 0     \\
           \end{array}
 \right].
$$
We consider the following three cases.

$\mathbf{Case (1)}$ Suppose that $c(\gamma)=d(\gamma)=0$.  If $g_1(\gamma)=g_2(\gamma)=0$, then $\mathrm{det}(\mathbf{T}(\gamma))=0$ implies that $\Rank(\mathbf{D}(\gamma))\leq 3$. Thus, we have  $g_1(\gamma)\neq 0$ or $g_2(\gamma)\neq 0$. Similarly, if $a(\gamma)=b(\gamma)=0$, then $\Rank(\mathbf{D}(\gamma))\leq 3$. Hence, $a(\gamma)\neq 0$ or $b(\gamma)\neq0$. Since ${\mathbf{D}}^2=0$, we have
\begin{equation*}
(ag_1E+ag_2G)(\gamma)=0,
\end{equation*}
\begin{equation*}
(ag_1F+ag_2H)(\gamma)=0,
\end{equation*}
\begin{equation*}
(bg_1E+bg_2G)(\gamma)=0,
\end{equation*}
\begin{equation*}
(bg_1F+bg_2H)(\gamma)=0.
\end{equation*}.

Then we obtain $\mathrm{det}\left[ \begin{array}{ccc}
 E(\gamma)  & F(\gamma)    \\
 G(\gamma)  & H(\gamma)  \\
 \end{array}
 \right]=0$, so the third row and the fourth row of $\mathbf{D}(\gamma)$ are linearly dependent. Therefore, $\Rank(\mathbf{D}(\gamma))\leq 3$, so we are done.

$\mathbf{Case (2)}$ Either $c(\gamma)\neq 0$ or $d(\gamma)\neq 0$. $\WLOG$ suppose that $c(\gamma)=0$ and $d(\gamma)\neq 0$.
Then, we have
$$\mathbf{D}(\gamma)=\left[ \begin{array}{cccccccc}
0 & 0 & (a g_1)(\gamma) & (a g_2)(\gamma)  & A(\gamma)        & B(\gamma)         & p_{17}(\gamma)   & p_{18}(\gamma) \\
0 & 0 & (b g_1)(\gamma) & (b g_2)(\gamma)  & C(\gamma)        & D(\gamma)         & p_{27}(\gamma)   & p_{28}(\gamma) \\
0 & 0 & 0               & 0                & \mathbf{0}       & -(d g_2)(\gamma)  & E(\gamma)        & F(\gamma)\\
0 & 0 & 0               & 0                & \mathbf{0}       & (d g_1)(\gamma)   & G(\gamma)        & H(\gamma)\\
0 & 0 & 0               & 0                & 0                & 0                 & -(e d)(\gamma)   & -(f d)(\gamma)\\
0 & 0 & 0               & 0                & 0                & 0                 & \mathbf{0}       & \mathbf{0}    \\
0 & 0 & 0               & 0                & 0                & 0                 & 0                & 0     \\
0 & 0 & 0               & 0                & 0                & 0                 & 0                & 0     \\
           \end{array}
 \right].
$$

$\mathbf{Case (2)(i)}$ Suppose that $g_1(\gamma)=g_2(\gamma)=0$. Since $\mathbf{D}^2=0$, we have $(edA)(\gamma)=0$, $(fdA)(\gamma)=0$, $(edC)(\gamma)=0$ and $(fdC)(\gamma)=0$. Then we get ${e}(\gamma)=0=f(\gamma)$ or $A(\gamma)=0=C(\gamma)$. If  ${e}(\gamma)=0=f(\gamma)$, then $\mathrm{det}(\mathbf{T}(\gamma))=0$ implies that $\Rank(\mathbf{D}(\gamma))\leq 3$. If $A(\gamma)=0=C(\gamma)$, then we only have three non-zero columns, so we are done. If $a(\gamma)=b(\gamma)=0$, using the similar argument, we get $\Rank(\mathbf{D}(\gamma))\leq 3$, which is a contradiction.

$\mathbf{Case (2)(ii)}$ Consider the case $g_1(\gamma)\neq 0$ or $g_2(\gamma)\neq 0$. $\WLOG$ suppose that $g_1(\gamma)=0$ and $g_2(\gamma)\neq 0$. Then we have
\begin{equation*}
    \mathbf{D}(\gamma) = \begin{tikzpicture}[baseline={(m.center)}]
            \matrix [matrix of math nodes,left delimiter={[ },right delimiter={ ] }] (m)
            {
0 & 0 &\mathbf{0} & a g_2            & A               & B           & p_{17}       & p_{18} \\
0 & 0 &\mathbf{0} & b g_2            & C               & D           & p_{27}       & p_{28} \\
0 & 0 & 0         & 0                & 0               & -d g_2      & E            & F\\
0 & 0 & 0         & 0                & 0               & \mathbf{0}  & G            & H\\
0 & 0 & 0         & 0                & 0               & 0           & -e d         & -f d\\
0 & 0 & 0         & 0                & 0               & 0           & 0            & 0    \\
0 & 0 & 0         & 0                & 0               & 0           & 0            & 0     \\
0 & 0 & 0         & 0                & 0               & 0           & 0            & 0     \\
};
            \draw (m-1-4.north west) rectangle (m-2-5.south east-|m-1-5.north east);
            \draw (m-4-7.north west) rectangle (m-5-8.south east-|m-4-8.north east);
            \end{tikzpicture} (\gamma).
\end{equation*}
Since $\mathbf{D}^2=0$, we get
\begin{equation}\label{Eq1}
\left(\,{a}g_2G-e {d} A   \right)(\gamma)=0,
\end{equation}
\begin{equation}\label{Eq2}
\left(\,{a} g_2H-f {d} A   \right)(\gamma)=0,
\end{equation}
\begin{equation}\label{Eq3}
\left(\,{b} g_2G-e{d} C   \right)(\gamma)=0,
\end{equation}
\begin{equation}\label{Eq4}
\left(\, {b} g_2H-f {d} C   \right)(\gamma)=0.
\end{equation}

If $a(\gamma)=b(\gamma)=0$, then we have $(edA)(\gamma)=0$, $(fdA)(\gamma)=0$, $(edC)(\gamma)=0$ and $(fdC)(\gamma)=0$. Then we get ${e}(\gamma)=0=f(\gamma)$ or $A(\gamma)=0=C(\gamma)$. In each case, $\Rank\left(\mathbf{D(\gamma)}\right)\leq 3$, so we are done. Thus, $a(\gamma)\neq 0$ or $b(\gamma)\neq 0$.

If  $e(\gamma)=f(\gamma)=0$, then since $a(\gamma)\neq 0$ or $b(\gamma)\neq 0$, we have $G(\gamma)=0=H(\gamma)$. Clearly, we only have three non-zero rows, that means $\Rank\left(\mathbf{D(\gamma)}\right)\leq 3$. Thus, $e(\gamma)\neq 0$ or $f(\gamma)\neq 0$.

If $A(\gamma)=C(\gamma)=0$, then $G(\gamma)=0=H(\gamma)$, and vice versa. Since $\mathrm{det}(\mathbf{M}(\gamma))=0$, the rank of $\left(\left[ \begin{array}{ccc}
 ag_2  & p_{17}   & p_{18} \\
 bg_2  & p_{27}   & p_{28} \\
 0  & -ed        & -fd     \\
           \end{array}
 \right](\gamma) \right)$ is at most $2$. Thus, $\Rank\left(\mathbf{D(\gamma)}\right)\leq 3$ , which is a contradiction.

Hence we have $a(\gamma)\neq 0$ or $b(\gamma)\neq 0$, $e(\gamma)\neq 0$ or $f(\gamma)\neq 0$, $A(\gamma)\neq 0$ or $C(\gamma)\neq 0$, and $G(\gamma)\neq 0$ or $H(\gamma)\neq 0$. Then Equations (\ref{Eq1} and \ref{Eq2}) or (\ref{Eq3} and \ref{Eq4}) give us
\begin{equation}
\left(fG \right)(\gamma)=\left( eH \right)(\gamma).
\end{equation}
 Thus, the fourth and the the fifth rows of $\mathbf{D}(\gamma)$ are linearly dependent. Equations (\ref{Eq1} and \ref{Eq3}) or (\ref{Eq2} and \ref{Eq4}) give us
\begin{equation}
\left(bA\right)(\gamma)=\left(aC\right)(\gamma).
\end{equation}
  So, the fourth and the fifth columns of $\mathbf{D}(\gamma)$ are linearly dependent. By applying elementary row and column operations, the fifth row and the fourth column of $\mathbf{D}(\gamma)$ become zero. Since $\mathrm{det}(\mathbf{T}(\gamma))=0$, the rank of the matrix  $\mathbf{D}(\gamma)$ is at most $3$, which is a contradiction.

$\mathbf{Case (2)(iii)}$ Suppose that $g_1(\gamma)\neq 0$ and $g_2(\gamma)\neq 0$.
If we apply elementary row and column operations to the matrix $\mathbf{D}(\gamma)$, we get
 \begin{align*}
\mathbf{D} (\gamma)\,
& \xrightarrow[C_3 \to g_2C_3-g_1C_4, \, \, C_4\to \frac{C_4}{g_2}]{R_4\to {g_2}R_4+{g_1}R_3}
\mathbf{\tilde{D}}(\gamma) =\left[
\begin{array}{cccccccc}
0 & 0 & \mathbf{0}     &a   & A     & B       & p_{17}            & p_{18} \\
0 & 0 & \mathbf{0}     &b   & C     & D       & p_{27}            & p_{28} \\
0 & 0 & 0     & 0     & 0     & -d g_2  & E                   & F\\
0 & 0 & 0     & 0     & 0     & \mathbf{0}       & {g_2}G+{g_1}E  & {g_2}H+{g_1}F\\
0 & 0 & 0     & 0     & 0     & 0       & -e{d}      & -f{d}\\
0 & 0 & 0     & 0     & 0     & 0       & 0                   & 0  \\
0 & 0 & 0     & 0     & 0     & 0       & 0                   & 0  \\
0 & 0 & 0     & 0     & 0     & 0       & 0                   & 0  \\
\end{array}
\right](\gamma).
\end{align*}
 Clearly $\Rank(\mathbf{D}(\gamma))=\Rank(\mathbf{\tilde{D}}(\gamma))$ and $\mathrm{det}(\mathbf{T}(\gamma))=0=m_{5678}^{1234}(\mathbf{\tilde{D}}(\gamma))$. Since $\mathbf{D}(\gamma)$ and therefore $\mathbf{\tilde{D}}(\gamma)$ are square-zero matrices, we obtain
\begin{equation}\label{eqqqq1}
\left(\,{a}( g_2G+g_1 E)-e {d} A   \right)(\gamma)=0,
\end{equation}
\begin{equation}\label{eqqqq2}
\left(\,{a}( g_2H+g_1 F)-f {d} A   \right)(\gamma)=0,
\end{equation}
\begin{equation}\label{eqqqq3}
\left(\,{b}( g_2G+g_1 E)-e{d} C   \right)(\gamma)=0,
\end{equation}
\begin{equation}\label{eqqqq4}
\left(\, {b}( g_2H+g_1 F)-f {d} C   \right)(\gamma)=0.
\end{equation}

If $a(\gamma)=b(\gamma)=0$, then we have $(edA)(\gamma)=0$, $(fdA)(\gamma)=0$, $(edC)(\gamma)=0$ and $(fdC)(\gamma)=0$. Then we get ${e}(\gamma)=0=f(\gamma)$ or $A(\gamma)=0=C(\gamma)$. In each case, $\Rank(\mathbf{\tilde{D}}(\gamma))\leq 3$, so we are done. Thus, $a(\gamma)\neq 0$ or $b(\gamma)\neq 0$.

If  $e(\gamma)=f(\gamma)=0$, then since $a(\gamma)\neq 0$ and $b(\gamma)\neq 0$, we have $(g_2G+g_1E)(\gamma)=0=(g_2H+g_1F)(\gamma)$. Clearly, we only have three non-zero rows, that means $\Rank(\mathbf{\tilde{D}}(\gamma))\leq 3$. Thus, $e(\gamma)\neq 0$ or $f(\gamma)\neq 0$.

If $A(\gamma)=C(\gamma)=0$, then $(g_2G+g_1E)(\gamma)=0=(g_2H+g_1F)(\gamma)$ and vice versa. Then we have
\begin{equation*}
    \mathbf{\tilde{D}}(\gamma) = \begin{tikzpicture}[baseline={(m.center)}]
            \matrix [matrix of math nodes,left delimiter={[},right delimiter={]}] (m)
            {
0 & 0 & \mathbf{0}     &a      & \mathbf{0}  & B           & p_{17}            & p_{18} \\
0 & 0 & \mathbf{0}     &b      & \mathbf{0}  & D           & p_{27}            & p_{28} \\
0 & 0 & 0              & 0     & 0           & -d g_2      & E            & F\\
0 & 0 & 0              & 0     & 0           & \mathbf{0}  & \mathbf{0}   & \mathbf{0}\\
0 & 0 & 0              & 0     & 0           & 0           & -e{d}        & -f{d}\\
0 & 0 & 0              & 0     & 0           & 0           & 0                   & 0  \\
0 & 0 & 0              & 0     & 0           & 0           & 0                   & 0  \\
0 & 0 & 0              & 0     & 0           & 0           & 0                   & 0  \\
            };
            \draw (m-1-4.north west) rectangle (m-2-4.south east-|m-1-4.north east);
            \draw (m-5-4.north west) rectangle (m-5-4.south east-|m-4-4.north east);
            \draw (m-5-7.north west) rectangle (m-5-8.south east-|m-5-8.north east);
            \draw (m-1-7.north west) rectangle (m-2-7.south east-|m-1-8.north east);
            \end{tikzpicture} (\gamma).
\end{equation*}
Since $\mathrm{det}(\mathbf{M}(\gamma))=0$, the rank of $\left(\left[ \begin{array}{ccc}
 a  & p_{17}   & p_{18} \\
 b  & p_{27}   & p_{28} \\
 0  & -ed        & -fd     \\
           \end{array}
 \right](\gamma) \right)$ is at most $2$. Thus, $\Rank(\mathbf{\tilde{D}}(\gamma))\leq 3$, which is a contradiction.

Hence we have $a(\gamma)\neq 0$ or $b(\gamma)\neq 0$, $e(\gamma)\neq 0$ or $f(\gamma)\neq 0$, $A(\gamma)\neq 0$ or $C(\gamma)\neq 0$, and $(g_2G+g_1E)(\gamma)\neq 0$ or $(g_2H+g_1F)(\gamma)\neq 0$.
\begin{equation*}
    \mathbf{\tilde{D}}(\gamma) = \begin{tikzpicture}[baseline={(m.center)}]
            \matrix [matrix of math nodes,left delimiter={[},right delimiter={]}] (m)
            {
0 & 0 & \mathbf{0}     &a      & A        & B           & p_{17}            & p_{18} \\
0 & 0 & \mathbf{0}     &b      & C        & D           & p_{27}            & p_{28} \\
0 & 0 & 0              & 0     & 0        & -d g_2      & E            & F\\
0 & 0 & 0              & 0     & 0        & \mathbf{0}  & {g_2}G+{g_1}E  & {g_2}H+{g_1}F\\
0 & 0 & 0              & 0     & 0        & 0           & -e{d}        & -f{d}\\
0 & 0 & 0              & 0     & 0        & 0           & 0                   & 0  \\
0 & 0 & 0              & 0     & 0        & 0           & 0                   & 0  \\
0 & 0 & 0              & 0     & 0        & 0           & 0                   & 0  \\
            };
            \draw (m-1-4.north west) rectangle (m-2-5.south east-|m-1-5.north east);
            \draw (m-4-7.north west) rectangle (m-5-8.south east-|m-4-8.north east);
            \end{tikzpicture} (\gamma).
\end{equation*}
Then Equations (\ref{eqqqq1} and \ref{eqqqq2}) or (\ref{eqqqq3} and \ref{eqqqq4}) give us
\begin{equation}
\left({f}(g_2G+g_1E)\right)(\gamma)=\left({e}(g_2H+g_1F)\right)(\gamma).
\end{equation}
 Thus, the fourth and the the fifth rows of $\mathbf{\tilde{D}}(\gamma)$ are linearly dependent. Equations (\ref{eqqqq1} and \ref{eqqqq3}) or (\ref{eqqqq2} and \ref{eqqqq4}) give us
\begin{equation}
\left(bA\right)(\gamma)=\left(aC\right)(\gamma).
\end{equation}
  So, the fourth and the fifth columns of $\mathbf{\tilde{D}}(\gamma)$ are linearly dependent. By applying elementary row and column operations, the fifth row and the fourth column of $\mathbf{\tilde{D}}(\gamma)$ become zero. Since $\mathrm{det}(\mathbf{T}(\gamma))=0$, the rank of the matrix  $\mathbf{\tilde{D}}(\gamma)$ is at most $3$, which is a contradiction.

$\mathbf{Case (3)}$ Suppose that $c(\gamma)\neq 0$ and $d(\gamma)\neq 0$.

$\mathbf{Case (3)(i)}$Suppose that $g_1(\gamma)=g_2(\gamma)=0$, then we have

\begin{align*}
\mathbf{D}(\gamma)=\left[ \begin{array}{cccccccc}
0 & 0 & \mathbf{0} & \mathbf{0}  & A           & B            & p_{17}   & p_{18} \\
0 & 0 & \mathbf{0} & \mathbf{0}  & C           & D            & p_{27}   & p_{28} \\
0 & 0 & 0          & 0           & \mathbf{0}  & \mathbf{0}   & E        & F\\
0 & 0 & 0          & 0           & \mathbf{0}  & \mathbf{0}   & G        & H\\
0 & 0 & 0          & 0           & 0           & 0            & -e d     & -f d\\
0 & 0 & 0          & 0           & 0           & 0            & e c      & f c  \\
0 & 0 & 0          & 0           & 0           & 0            & 0        & 0     \\
0 & 0 & 0          & 0           & 0           & 0            & 0        & 0     \\
\end{array}
\right](\gamma).
\end{align*}
If $e(\gamma)=f(\gamma)=0$, then $\mathrm{det}(\mathbf{T}(\gamma))=0$ implies that $\Rank(\mathbf{D}(\gamma))\leq 3$. If $e(\gamma)\neq 0$ or $f(\gamma)\neq 0$, then $\mathbf{D}^2=0$ implies that the rank of $\left(\left[ \begin{array}{ccc}
 A  & B    \\
 C  & D  \\
 \end{array}
 \right](\gamma)\right)$ is at most $1$. Thus, $\Rank(\mathbf{D}(\gamma))\leq 3$, so we are done. 

$\mathbf{Case (3)(ii)}$ Consider the case $g_1(\gamma)\neq 0$ or $g_2(\gamma)\neq 0$. $\WLOG$ suppose that $g_1(\gamma)=0$ and $g_2(\gamma)\neq 0$. If we apply elementary row and column operations to the matrix $\mathbf{D}(\gamma)$, we get
 \begin{align*}
\mathbf{D} (\gamma)\,
& \xrightarrow[C_5 \to dC_5-cC_6]{R_6\to {d}R_6+{c}R_5, \, \, R_5\to \frac{R_5}{d}}
\mathbf{\dot{D}}(\gamma) =\left[
\begin{array}{cccccccc}
0 & 0 & \mathbf{0}     &ag_2   & dA-cB         & B          & p_{17}            & p_{18} \\
0 & 0 & \mathbf{0}     &bg_2   & dC-cD         & D          & p_{27}            & p_{28} \\
0 & 0 & 0              & 0     & \mathbf{0}    & -d g_2     & E                 & F\\
0 & 0 & 0              & 0     & \mathbf{0}    & \mathbf{0} & G                 & H\\
0 & 0 & 0              & 0     & 0             & 0          & -e                & -f\\
0 & 0 & 0              & 0     & 0             & 0          & \mathbf{0}        & \mathbf{0}  \\
0 & 0 & 0              & 0     & 0             & 0          & 0                 & 0  \\
0 & 0 & 0              & 0     & 0             & 0          & 0                 & 0  \\
\end{array}
\right](\gamma).
\end{align*}

Since $\mathbf{\dot{D}}^2=0$, we get
\begin{equation}\label{EQ1}
\left(\,{a}g_2G-e({d} A-cB  ) \right)(\gamma)=0,
\end{equation}
\begin{equation}\label{EQ2}
\left(\,{a} g_2H-f({d} A-cB)   \right)(\gamma)=0,
\end{equation}
\begin{equation}\label{EQ3}
\left(\,{b} g_2G-e({d} C-cD)   \right)(\gamma)=0,
\end{equation}
\begin{equation}\label{EQ4}
\left(\, {b} g_2H-f( {d} C-cD   )\right)(\gamma)=0.
\end{equation}

If $a(\gamma)=b(\gamma)=0$, then we have $(e(dA-cB))(\gamma)=0$, $(f(dA-cB))(\gamma)=0$, $(e(dC-cD))(\gamma)=0$ and $(f(dC-cD))(\gamma)=0$. Then we get ${e}(\gamma)=0=f(\gamma)$ or $(dA-cB)(\gamma)=0=(dC-cD)(\gamma)$. In each case, $\Rank\left( \mathbf{\dot{D}}(\gamma) \right)\leq 3$, so we are done. Thus, $a(\gamma)\neq 0$ or $b(\gamma)\neq 0$.

If  $e(\gamma)=f(\gamma)=0$, then since $a(\gamma)\neq 0$ or $b(\gamma)\neq 0$, we have $G(\gamma)=0=H(\gamma)$. Clearly, we only have three non-zero rows, that means $\Rank\left( \mathbf{\dot{D}}(\gamma)\right)\leq 3$. Thus, $e(\gamma)\neq 0$ or $f(\gamma)\neq 0$.

If $(dA-cB)(\gamma)=(dC-cD)(\gamma)=0$, then $G(\gamma)=0=H(\gamma)$, and vice versa. Since $\mathrm{det}(\mathbf{M}(\gamma))=0$, the rank of $\left(\left[ \begin{array}{ccc}
 ag_2  & p_{17}   & p_{18} \\
 bg_2  & p_{27}   & p_{28} \\
 0     & -e        & -f     \\
           \end{array}
 \right](\gamma) \right)$ is at most $2$. Thus, $\Rank\left(\mathbf{\dot{D}}(\gamma)\right)\leq 3$ , which is a contradiction.

 Hence we have $a(\gamma)\neq 0$ or $b(\gamma)\neq 0$, $e(\gamma)\neq 0$ or $f(\gamma)\neq 0$, $(dA-cB)(\gamma)\neq 0$ or $(dC-cD)(\gamma)\neq 0$, and $G(\gamma)\neq 0$ or $H(\gamma)\neq 0$. Then Equations (\ref{EQ1} and \ref{EQ2}) or (\ref{EQ3} and \ref{EQ4}) give us
\begin{equation}
\left(fG \right)(\gamma)=\left( eH \right)(\gamma).
\end{equation}
 Thus, the fourth and the the fifth rows of $\mathbf{D}(\gamma)$ are linearly dependent. Equations (\ref{EQ1} and \ref{EQ3}) or (\ref{EQ2} and \ref{EQ4}) give us
\begin{equation}
\left(b(dA-cB)\right)(\gamma)=\left(a(dC-cD)\right)(\gamma).
\end{equation}
  So, the fourth and the fifth columns of $\mathbf{\dot{D}}(\gamma)$ are linearly dependent. By applying elementary row and column operations, the fifth row and the fourth column of $\mathbf{\dot{D}}(\gamma)$ become zero. Since $\mathrm{det}(\mathbf{T}(\gamma))=0$, the rank of the matrix  $\mathbf{\dot{D}}(\gamma)$ is at most $3$, which is a contradiction.

$\mathbf{Case (3)(iii)}$ Suppose that $g_1(\gamma)\neq 0$ and $g_2(\gamma)\neq0$. If we apply elementary row and column operations to the matrix $\mathbf{\tilde{D}}(\gamma)$, we get
 \begin{align*}
\mathbf{\tilde{D}} (\gamma)\,
& \xrightarrow[C_5 \to dC_5-cC_6]{R_6\to {d}R_6+{c}R_5, \, \, R_5\to \frac{R_5}{d}}
\mathbf{\acute{D}}(\gamma) =\left[
\begin{array}{cccccccc}
0 & 0 & \mathbf{0}     &a      & dA-cB         & B          & p_{17}            & p_{18} \\
0 & 0 & \mathbf{0}     &b      & dC-cD         & D          & p_{27}            & p_{28} \\
0 & 0 & 0              & 0     & \mathbf{0}    & -d g_2     & E                 & F\\
0 & 0 & 0              & 0     & \mathbf{0}    & \mathbf{0} & {g_2}G+{g_1}E     & {g_2}H+{g_1}F\\
0 & 0 & 0              & 0     & 0             & 0          & -e             & -f\\
0 & 0 & 0              & 0     & 0             & 0          & \mathbf{0}        & \mathbf{0}  \\
0 & 0 & 0              & 0     & 0             & 0          & 0                 & 0  \\
0 & 0 & 0              & 0     & 0             & 0          & 0                 & 0  \\
\end{array}
\right](\gamma).
\end{align*}
 Clearly $\Rank(\mathbf{D}(\gamma))=\Rank(\mathbf{\acute{D}}(\gamma))$ and $\mathrm{det}(\mathbf{T}(\gamma))=0=m_{5678}^{1234}(\mathbf{\acute{D}}(\gamma))$. Since $\mathbf{D}(\gamma)$ and $\mathbf{\acute{D}}(\gamma)$ are square-zero matrices, we obtain
\begin{equation}\label{eqqq5}
\left(\,{a}( g_2G+g_1 E)-e( {d} A- cB  ) \right)(\gamma)=0,
\end{equation}
\begin{equation}\label{eqqq6}
\left(\,{a}( g_2H+g_1 F)-f( {d} A- cB  ) \right)(\gamma)=0,
\end{equation}
\begin{equation}\label{eqqq7}
\left(\,{b}( g_2G+g_1 E)-e( {d} C - cD ) \right)(\gamma)=0,
\end{equation}
\begin{equation}\label{eqqq8}
\left(\, {b}( g_2H+g_1 F)-f({d} C- cD ) \right)(\gamma)=0.
\end{equation}

If $a(\gamma)=b(\gamma)=0$, then we have $(e(dA-cB))(\gamma)=0$, $(f(dA-cB))(\gamma)=0$, $(e(dC-cD))(\gamma)=0$ and $(f(dC-cD))(\gamma)=0$. Then we get ${e}(\gamma)=0=f(\gamma)$ or $(dA-cB)(\gamma)=0=(dC-cD)(\gamma)$. In each case, $\Rank(\mathbf{\acute{D}}(\gamma))\leq 3$, so we are done. Thus, $a(\gamma)\neq 0$ and $b(\gamma)\neq 0$.

If $e(\gamma)=f(\gamma)=0$, then since $a(\gamma)\neq 0$ and $b(\gamma)\neq 0$, we have $(g_2G+g_1E)(\gamma)=0=(g_2H+g_1F)(\gamma)$. Clearly, we only have three non-zero rows. Then $\Rank(\mathbf{\acute{D}}(\gamma))\leq 3$. Thus, $e(\gamma)\neq 0$ or $f(\gamma)\neq 0$.

If $(dA-cB)(\gamma)=(dC-cD)(\gamma)=0$, then again $(g_2G+g_1E)(\gamma)=0=(g_2H+g_1F)(\gamma)$ and vice versa.
Since $\mathrm{det}(\mathbf{M}(\gamma))=0$, the rank of $\left(\left[ \begin{array}{ccc}
 a  & p_{17}   & p_{18} \\
 b  & p_{27}   & p_{28} \\
 0  & -e       & -f     \\
           \end{array}
 \right](\gamma) \right)$ is at most $2$. Thus, $\Rank(\mathbf{\acute{D}}(\gamma))\leq 3$, which is a contradiction.

 Hence we can assume that $a(\gamma)\neq 0$ or $b(\gamma)\neq 0$, $e(\gamma)\neq 0$ or $f(\gamma)\neq 0$, $(dA-cB)(\gamma)\neq 0$ or $(dC-cD)(\gamma)\neq 0$, and $(g_2G+g_1E)(\gamma)\neq 0$ or $(g_2H+g_1F)(\gamma)\neq 0$. Then Equations (\ref{eqqq5} and \ref{eqqq6}) or (\ref{eqqq7} and \ref{eqqq8}) give us
\begin{equation}\label{eqqq9}
\left({f}(g_2G+g_1E)\right)(\gamma)=\left({e}(g_2H+g_1F)\right)(\gamma),
\end{equation}

 Thus, the fourth and the the fifth rows of $\mathbf{\acute{D}}(\gamma)$ are linearly dependent. Equations (\ref{eqqq5} and \ref{eqqq7}) or (\ref{eqqq6} and \ref{eqqq8}) give us
\begin{equation}\label{eqqq10}
\left(b(dA-cB)\right)(\gamma)=\left(a(dC-cD)\right)(\gamma).
\end{equation}
  So, the fourth and the fifth columns of $\mathbf{\acute{D}}(\gamma)$ are linearly dependent. By applying elementary row and column operations, the fifth row and the fourth column of $\mathbf{\acute{D}}(\gamma)$ become zero. Since $\mathrm{det}(\mathbf{T}(\gamma))=0$, the rank of the matrix  $\mathbf{\acute{D}}(\gamma)$ is at most $3$, which is a contradiction.
\end{proof}\begin{proposition}\label{prop7}
If $\sigma \leq (12)(36)(45)(78)$, then $r<4$.
\end{proposition}
\begin{proof}
Suppose to the contrary that $r\geq 4$, and by Remark \ref{remarkvariable}, we may take $r=4$. We have
$$\mathbf{D}=\left[ \begin{array}{cccccccc}
                      0 & p_{12}& p_{13} & p_{14} & p_{15} & p_{16} & p_{17} & p_{18} \\
                      0 & 0     & 0      & 0      & p_{25} & p_{26} & p_{27} & p_{28} \\
                      0 & 0     & 0      & 0      & p_{35} & p_{36} & p_{37} & p_{38} \\
                      0 & 0     & 0      & 0      & p_{45} & p_{46} & p_{47} & p_{48} \\
                      0 & 0     & 0      & 0      & 0      & 0      & 0      & p_{58} \\
                      0 & 0     & 0      & 0      & 0      & 0      & 0      & p_{68} \\
                      0 & 0     & 0      & 0      & 0      & 0      & 0      & p_{78} \\
                      0 & 0     & 0      & 0      & 0      & 0      & 0      & 0      \\
                    \end{array}
 \right].
$$
The matrix $\mathbf{D}$ is of type $(1,3,3,1)$, and by Remark \ref{remark2} there exists a free flag with $l=3$. However, Theorem \ref{Theoremrandl} implies that $l\geq r$, so $l\geq 4$, which is a contradiction.
\end{proof}
\begin{proposition}\label{prop8}
If $\sigma \leq (12)(38)(47)(56)$  or  $\sigma \leq (16)(25)(34)(78)$, then $r<3$.
\end{proposition}
\noindent These cases are symmetric, so it is enough to prove the case $\sigma \leq (12)(38)(47)(56)$.
\begin{proof}
Suppose to the contrary that $\sigma \leq (12)(38)(47)(56)$ and $r\geq3$. By Remark \ref{remarkvariable}, it is enough to consider $r=3$.
$$\mathbf{D}=\left[ \begin{array}{cccccccc}
                      0 & p_{12}& p_{13} & p_{14} & p_{15} & p_{16} & p_{17} & p_{18} \\
                      0 & 0     & 0      & 0      & 0      & p_{26} & p_{27} & p_{28} \\
                      0 & 0     & 0      & 0      & 0      & p_{36} & p_{37} & p_{38} \\
                      0 & 0     & 0      & 0      & 0      & p_{46} & p_{47} & p_{48} \\
                      0 & 0     & 0      & 0      & 0      & p_{56} & p_{57} & p_{58} \\
                      0 & 0     & 0      & 0      & 0      & 0      & 0      & 0      \\
                      0 & 0     & 0      & 0      & 0      & 0      & 0      & 0      \\
                      0 & 0     & 0      & 0      & 0      & 0      & 0      & 0      \\
                    \end{array}
 \right].
$$
 The matrix $\mathbf{D}$ is of type $(1,4,3)$, and by Remark \ref{remark2} there exists a free flag with $l=2$. However, Theorem \ref{Theoremrandl} implies that $l\geq r$, so $l\geq 3$, which is a contradiction.
\end{proof}
\begin{proposition}\label{prop9}
If $\sigma \leq (14)(23)(58)(67)$, then $r<3$.
\end{proposition}
\begin{proof}
Suppose to the contrary that $r\geq 3$. By Remark \ref{remarkvariable}, we may take $r=3$ and consider
$$\mathbf{D}=\left[ \begin{array}{cccccccc}
                      0 & 0 & p_{13} & p_{14} & p_{15} & p_{16} & p_{17} & p_{18} \\
                      0 & 0 & p_{23} & p_{24} & p_{25} & p_{26} & p_{27} & p_{28} \\
                      0 & 0 & 0      & 0      & 0      & 0      & p_{37} & p_{38} \\
                      0 & 0 & 0      & 0      & 0      & 0      & p_{47} & p_{48} \\
                      0 & 0 & 0      & 0      & 0      & 0      & p_{57} & p_{58} \\
                      0 & 0 & 0      & 0      & 0      & 0      & p_{67} & p_{68} \\
                      0 & 0 & 0      & 0      & 0      & 0      & 0      & 0     \\
                      0 & 0 & 0      & 0      & 0      & 0      & 0      & 0      \\
                    \end{array}
 \right].
$$
The matrix $\mathbf{D}$ is of type $(2,4,2)$, and by Remark \ref{remark2} there exists a free flag with $l=2$. However, Theorem \ref{Theoremrandl} implies that $l\geq 3$, which is a contradiction.
\end{proof}\begin{proposition}\label{prop10}
If $\sigma \leq (18)(23)(47)(56)$ or $\sigma \leq (18)(25)(34)(67)$, then $r<2$.
\end{proposition}
\noindent Note that these cases are symmetric, so it is enough to prove the case ${\sigma \leq (18)(23)(47)(56)}$.
\begin{proof}
Suppose not. By Remark \ref{remarkvariable}, we may take $r=2$.
$$\mathbf{D}=\left[ \begin{array}{cccccccc}
                      0 & 0      & p_{13} & p_{14} & p_{15} & p_{16} & p_{17} & p_{18} \\
                      0 & 0      & p_{23} & p_{24} & p_{25} & p_{26} & p_{27} & p_{28} \\
                      0 & 0      & 0      & 0      & 0      & p_{36} & p_{37} & p_{38} \\
                      0 & 0      & 0      & 0      & 0      & p_{46} & p_{47} & p_{48} \\
                      0 & 0      & 0      & 0      & 0      & p_{56} & p_{57} & p_{58} \\
                      0 & 0      & 0      & 0      & 0      & 0      & 0      & 0 \\
                      0 & 0      & 0      & 0      & 0      & 0      & 0      & 0 \\
                      0 & 0      & 0      & 0      & 0      & 0      & 0      & 0      \\
                    \end{array}
 \right].
$$
Define $$ \mathbf{L}:=\left[ \begin{array}{ccc}
p_{13} & p_{14} & p_{15} \\
p_{23} & p_{24} & p_{25}\\
\end{array}
\right] \mbox{ and } \, \mathbf{N}:=\left[ \begin{array}{ccc}
 p_{36} & p_{37} & p_{38} \\
p_{46} & p_{47} & p_{48} \\
p_{56} & p_{57} & p_{58}
\end{array}
\right]=\left[ \begin{array}{ccc}
 A & D & G \\
 B & E & H \\
 C & F & I
\end{array}
\right].
$$
For all $\gamma \in  \mathbb{P}^{1}_k$, if $\Rank (\mathbf{N}(\gamma))\leq 1$, then $\Rank (\mathbf{D}(\gamma))\leq 3$. If $\Rank (\mathbf{N}(\gamma))=3$, then $\mathbf{LN}=0$ implies that $\Rank (\mathbf{L}(\gamma))=0$, which gives us  $\Rank (\mathbf{D}(\gamma))\leq 3$. Thus, for all $\gamma \in  \mathbb{P}^{1}_k$, we have $\Rank (\mathbf{N}(\gamma))=2$. Since $\Rank (\mathbf{N}(\gamma))=2$, $\Rank (\mathbf{L}(\gamma))\leq 1$. Note that $\Rank (\mathbf{L}(\gamma))$ cannot be zero, otherwise $\Rank (\mathbf{D}(\gamma))\leq 3$. Hence for all $\gamma \in  \mathbb{P}^{1}_k$, $\Rank (\mathbf{L}(\gamma))=1$.


Let $\mathbf{U}:=\left[ \begin{array}{ccc}
 1 & 0 & 1 \\
 0 & 1 & -1 \\
 1 & 1 & 0
\end{array}
\right].$ Since $\Rank \mathbf{N}\left([1:0]\right)=2= \Rank \mathbf{U}$, there exist nonsingular matrices $P$ and $Q$ in  $\Mat_{3\times 3}(k)$ such that $P \mathbf{N}\left( [1:0] \right)Q = \mathbf{U}$. Hence, $\WLOG$ we can assume that $\mathbf{N}\left([1:0]\right)=\mathbf{U}$ by replacing $ \mathbf{D}$ with the following matrix
$$ \left(\left[ \begin{array}{ccc}
  I_2 & 0    & 0\\
  0   & P    & 0 \\
  0   & 0    & Q^{-1}
\end{array}
\right]\, \mathbf{D} \, \left[ \begin{array}{ccc}
  I_2 & 0      & 0\\
  0   & P^{-1} & 0 \\
  0   & 0      & Q
\end{array}
\right]\right)$$  if necessary.

 By \cite[Lemma~$3.4$]{Refai1},
we can write
$$\mathbf{L}=\left[ \begin{array}{ccc}
\alpha h_1 & \alpha h_2 & \alpha h_3  \\
\beta h_1 & \beta h_2 & \beta h_3 \\
\end{array}
\right],
$$
where the entries are homogeneous polynomials, $\alpha$ and $\beta$  are relatively prime in $k[x_1,x_2]$. Then
\begin{equation*}
   \mathbf{D} = \begin{tikzpicture}[baseline={(m.center)}]
           \matrix [matrix of math nodes,left delimiter={[},right delimiter={]}] (m)
           {
                      0 & 0      & \alpha h_1 & \alpha h_2 & \alpha h_3 & p_{16} & p_{17} & p_{18} \\
                      0 & 0      & \beta h_1 & \beta h_2 & \beta h_3 & p_{26} & p_{27} & p_{28} \\
                      0 & 0      & 0      & 0      & 0      & A & D & G \\
                      0 & 0      & 0      & 0      & 0      & B & E & H \\
                      0 & 0      & 0      & 0      & 0      & C & F & I \\
                      0 & 0      & 0      & 0      & 0      & 0 & 0      & 0 \\
                      0 & 0      & 0      & 0      & 0      & 0 & 0      & 0 \\
                      0 & 0      & 0      & 0      & 0      & 0 & 0      & 0      \\
                     };
                     \draw(m-1-3.north west) rectangle (m-2-5.south east-|m-1-5.north east);
                   \draw(m-3-6.north west) rectangle (m-5-8.south east-|m-3-8.north east);
           \end{tikzpicture}, \mbox{ and define } \mathbf{M}:=\left[ \begin{array}{cccc}
                      \alpha & p_{16} & p_{17} & p_{18} \\
                      \beta  & p_{26} & p_{27} & p_{28} \\
                       0 & A & D & G \\
                       0 & B & E & H \\
                       0 & C & F & I
                    \end{array}
                    \right].
\end{equation*}
Furthermore, we define the homogeneous polynomials
\begin{align*}
f_1 & :=\alpha p_{26}-\beta p_{16},   & g_1 &  :=DH-EG,   &  g_2  & :=AH-BG,  & g_3   & :=AE-BD,\\
f_2 & := -\alpha p_{27}+\beta p_{17}, & g'_1&  :=DI-FG,   &  g'_2 & :=AI-CG,  & g'_3  & :=AF-CD,\\
f_3 & :=\alpha p_{28}-\beta p_{18},   & g''_1& :=EI-FH,  &  g''_2 & :=BI-CH,  & g''_3 & :=BF-CE.
\end{align*}
If $f_1$, $f_2$, $f_3$ are all zero at $[1:0]$, then the rank of the first two rows of $\mathbf{D}([1:0])$ will be at most $1$. Hence, $\Rank(\mathbf{D}([1:0]))\leq 3$ which is a contradiction. $\WLOG$ we can suppose that $f_1([1:0])$ is non-zero because if necessary, $ \mathbf{D}$ can be replaced by a matrix in the following form
$$ \left(\left[ \begin{array}{ccc}
  I_2 & 0    & 0\\
  0   & I_3    & 0 \\
  0   & 0    & W^{-1}
\end{array}
\right]\, \mathbf{D}([1:0]) \, \left[ \begin{array}{ccc}
  I_2 & 0      & 0\\
  0   & I_3 & 0 \\
  0   & 0      & W
\end{array}
\right]\right),$$ where $W$ is a permutation matrix in $\Mat_{3\times 3}(k)$. Now, notice that $f_1$ is a non-zero polynomial and moreover $g_i$, $g'_i$, and $g''_i$ are all non-zero polynomials for $i\in\{1,2,3\}$ because all $2\times 2$ minors of $\mathbf{U}W$ are non-zero. Then we have
\begin{equation}\label{eqqq20}
m_{1234}^{1234}(\mathbf{M})=f_1g_1+f_2g_2+f_3g_3, \, m_{1234}^{1235}(\mathbf{M})=f_1g'_1+f_2g'_2+f_3g'_3 \mbox{  and } m_{1234}^{1245}(\mathbf{M})=f_1g''_1+f_2g''_2+f_3g''_3.
\end{equation}

 Suppose that there exists $\gamma \in  \mathbb{P}^{1}_k$ such that $(m_{1234}^{1234}(\mathbf{M}))(\gamma)=0$.
 If $g_1(\gamma)\neq 0$ or $g_2(\gamma)\neq 0$ or $g_3(\gamma)\neq 0$, then  $\Rank \left[ \begin{array}{ccc}
 A & D & G  \\
 B & E & H \\
\end{array}
\right](\gamma)=2$. This implies that $\left[ \begin{array}{ccc}
 C & F & I  \\
\end{array}
\right](\gamma)$ is a linear combination of $\left[ \begin{array}{ccc}
 A & D & G  \\
\end{array}
\right](\gamma)$ and $\left[ \begin{array}{ccc}
 B & E & H  \\
\end{array}
\right](\gamma)$. Then the rank of the matrix $\mathbf{D}(\gamma)$ is at most $3$, which is a contradiction. Thus, for all $\gamma$

\begin{equation}\label{eq:idea}
\begin{alignedat}{2}
(m_{1234}^{1234}(\mathbf{M}))(\gamma)& =0 \Rightarrow  g_1(\gamma)=g_2(\gamma)=g_3(\gamma)=0, \\
(m_{1234}^{1235}(\mathbf{M}))(\gamma)& =0 \Rightarrow g'_1(\gamma)=g'_2(\gamma)=g'_3(\gamma)=0, \mbox{ and } \\
(m_{1234}^{1245}(\mathbf{M}))(\gamma)& =0 \Rightarrow  g''_1(\gamma)=g''_2(\gamma)=g''_3(\gamma)=0.
\end{alignedat}
\end{equation}

For $s\neq 0$, $t$ and $u$ in $k$, we can apply the following row and column operations to the matrix $\mathbf{D} \xrightarrow[c_4 \to ({c_4-uc_3})/s]{r_3\to r_3+ u r_4} \; \xrightarrow[c_5 \to c_5-tc_4]{r_4 \to s r_4+tr_5}  \mathbf{\tilde{D}}$, where

\begin{align*}
 \mathbf{\tilde{D}}=\left[ \begin{array}{cccccccc}
    0 & 0      & \alpha h_1 & \alpha \left( \frac{h_2-uh_1}{s}\right)  & \alpha \left(h_3-\frac{t}{s}(h_2-uh_1)\right) & p_{16} & p_{17} & p_{18} \\
    0 & 0      & \beta h_1  & \beta  \left( \frac{h_2-uh_1}{s}\right)  & \beta  \left(h_3-\frac{t}{s}(h_2-uh_1)\right) & p_{26} & p_{27} & p_{28} \\
    0 & 0      & 0          & 0                           & 0                                             & A+uB   & D+uE   & G+uH \\
    0 & 0      & 0          & 0                           & 0                                             & sB+tC  & sE+tF  & sH+tI \\
    0 & 0      & 0          & 0                           & 0                                             & C      & F      & I \\
    0 & 0      & 0          & 0                           & 0                                             & 0      & 0      & 0 \\
    0 & 0      & 0          & 0                           & 0                                             & 0      & 0      & 0 \\
    0 & 0      & 0          & 0                           & 0                                             & 0      & 0      & 0  \\
           \end{array}
 \right].
\end{align*}
 Clearly for all $\gamma$, $\Rank(\mathbf{D}(\gamma))=\Rank(\mathbf{\tilde{D}}(\gamma))$. Define the matrix
$$\mathbf{T}:=\left[ \begin{array}{cccc}
                      \alpha & p_{16} & p_{17} & p_{18} \\
                      \beta  & p_{26} & p_{27} & p_{28} \\
                       0 & A+uB   & D+uE   & G+uH \\
                       0 & sB+tC  & sE+tF  & sH+tI
                    \end{array}
                    \right].
$$
Note that
\begin{align*}
\mathrm{det}(\mathbf{T})& =f_1(sg_1+tg'_1+tug''_1)+f_2(sg_2+tg'_2+tug''_2)+ f_3(sg_3+tg'_3+tug''_3),\\
& =f_1m_{34}^{34}(\mathbf{T})+f_2m_{24}^{34}(\mathbf{T})+f_3m_{23}^{34}(\mathbf{T}).
\end{align*}

$\mathbf{Case (1)}$ There exist $s\neq 0$, $t$, $u$ in $k$ and there exists $\gamma \in  \mathbb{P}^{1}_k$ such that $\mathrm{det}{\mathbf{T}}(\gamma)=0$ and $(m_{34}^{34}(\mathbf{T}))(\gamma)\neq 0$ or $(m_{24}^{34}(\mathbf{T}))(\gamma)\neq 0$ or $(m_{23}^{34}(\mathbf{T}))(\gamma)\neq 0$. Under this assumption, we have
 $$\Rank \left[ \begin{array}{ccc}
 A+uB  &  D+uE  & G+uH  \\
 sB+tC &  sE+tF &sH+tI\\
\end{array}
\right](\gamma)=2.$$  This implies that $\left[ \begin{array}{ccc}
 C & F & I  \\
\end{array}
\right](\gamma)$ is a linear combination of $\left[ \begin{array}{ccc}
 A+uB & D+uE & G+uH  \\
\end{array}
\right](\gamma)$ and $\left[ \begin{array}{ccc}
 sB+tC & sE+tF & sH+tI \\
\end{array}
\right](\gamma)$. As a result, with the assumption of $\mathrm{det}{\mathbf{T}}(\gamma)=0$, the rank of the matrix  $\mathbf{\tilde{D}}(\gamma)$ is at most $3$, which is a contradiction.


$\mathbf{Case (2)}$ For all $s\neq 0$, $t$, $u$ in $k$ and for all $\gamma$ in  $\mathbb{P}^{1}_k$, $\mathrm{det}{\mathbf{T}}(\gamma)=0$ implies that $(m_{34}^{34}(\mathbf{T}))(\gamma)=(m_{24}^{34}(\mathbf{T}))(\gamma)=(m_{23}^{34}(\mathbf{T}))(\gamma)= 0$.

Remember that $(m_{34}^{34}(\mathbf{T}))(\gamma)=(m_{24}^{34}(\mathbf{T}))(\gamma)=(m_{23}^{34}(\mathbf{T}))(\gamma)= 0$ means that
\begin{align*}
(sg_1+tg'_1+tug''_1)(\gamma)=&0,\\
(sg_2+tg'_2+tug''_2)(\gamma)=&0, \mbox{ and }\\
(sg_3+tg'_3+tug''_3)(\gamma)=&0.
\end{align*}
Given any $\gamma \in \mathbb{P}^{1}_k$, define
$$v_1(\gamma):=( g_1(\gamma),g'_1(\gamma), g''_1(\gamma)), \,
v_2(\gamma):=( g_2(\gamma),g'_2(\gamma), g''_2(\gamma)  ), \,
v_3(\gamma):=( g_3(\gamma),g'_3(\gamma), g''_3(\gamma)  ).$$
Now consider $\mathrm{det}(\mathbf{T})$ in terms of minors of $\mathbf{M}$:
$$\mathrm{det}(\mathbf{T})= s(f_1g_1+f_2g_2+f_3g_3)+t(f_1g'_1+f_2g'_2+f_3g'_3)+tu(f_1g''_1+f_2g''_2+f_3g''_3).$$
Hence, $\forall \gamma$, $\mathrm{det}{\mathbf{T}}(\gamma)=0$ implies that
$s(m_{1234}^{1234}(\mathbf{M}))(\gamma)+t(m_{1234}^{1235}(\mathbf{M}))(\gamma)+tu (m_{1234}^{1245}(\mathbf{M}))(\gamma)=0.$\\

\noindent Let $\mathbf{m}(\gamma):=\left[(m_{1234}^{1234}(\mathbf{M}))(\gamma), (m_{1234}^{1235}(\mathbf{M}))(\gamma), (m_{1234}^{1245}(\mathbf{M}))(\gamma) \right]$. By the Statement (\ref{eq:idea}), for all $i\in \{1,2,3\}$, we get
$$
\mathbf{m}(\gamma)\, {\left[ \begin{array}{ccc}
 s & t & tu\\
\end{array}
\right]}^{\mathrm{T}}=0 \, \Rightarrow \, v_i(\gamma) \, {\left[ \begin{array}{ccc}
 s & t & tu\\
\end{array}
\right]}^{\mathrm{T}}=0.
$$
By permuting the roles of $s$, $t$ and $tu$,
$\mathbf{m}(\gamma)\,v^{\mathrm{T}}=0 \Rightarrow v_i(\gamma)\, v^{\mathrm{T}}=0 \mbox{ for all } v \in k^3.$
Hence, the null space of $\mathbf{m}(\gamma)$ is a subspace of the null space of $v_i(\gamma)$ for all $i \in \{1,2,3\}$, so
\begin{equation}\label{eqqq21}
 \mathrm{Null}(\mathbf{m}(\gamma))\subseteq \mathrm{Null} \left({\left[ \begin{array}{ccc}
 v_1(\gamma) & v_2(\gamma) & v_3(\gamma)\\
\end{array}
\right]}^{\mathrm{T}}\right).
\end{equation}
Since $\Rank(\mathbf{m}(\gamma))\leq 1$, we have $\mathrm{nullity}(\mathbf{m}(\gamma))\geq 2$. By Equation \ref{eqqq21},
$$\mathrm{nullity} \left({\left[ \begin{array}{ccc}
 v_1(\gamma) & v_2(\gamma) & v_3(\gamma)\\
\end{array}
\right]}^{\mathrm{T}}\right) \geq \mathrm{nullity}(\mathbf{m}(\gamma))\geq 2.$$
Then for all $\gamma$, $\Rank \left({\left[ \begin{array}{ccc}
 v_1(\gamma) & v_2(\gamma) & v_3(\gamma)\\
\end{array}
\right]}^{\mathrm{T}}\right) \leq 1$, so that each pair of $\{ v_1(\gamma), v_2(\gamma), v_3(\gamma)\}$ is linearly dependent. Since $\mathbf{N}\left([1:0]\right)=\mathbf{U}$, each $g_i$, $g'_i$ and $g''_i$ is non-zero polynomial for all $i\in \{1,2,3\}$.
Thus, as rational functions
\[
 \dfrac{g_i}{g_j}=\dfrac{g'_i}{g'_j}=\dfrac{g''_i}{g''_j} \mbox{, where  } i \neq j.
\]

Let $\lambda$ and $\kappa$ be rational functions such that
\[ \lambda:=\dfrac{g_1}{g_2}=\dfrac{g'_1}{g'_2}=\dfrac{g''_1}{g''_2}, \mbox{ and } \kappa:=\dfrac{g_1}{g_3}=\dfrac{g'_1}{g'_3}=\dfrac{g''_1}{g''_3}.
\]
Then,
\begin{align*}f_1g_1+f_2g_2+f_3g_3=g_1\left( f_1+f_2\frac{1}{\lambda}+f_3\frac{1}{\kappa}\right)=g_2\lambda \left( f_1+f_2\frac{1}{\lambda}+f_3\frac{1}{\kappa}\right)
=g_3\kappa\left( f_1+f_2\frac{1}{\lambda}+f_3\frac{1}{\kappa}\right),\\
f_1g'_1+f_2g'_2+f_3g'_3=g'_1\left( f_1+f_2\frac{1}{\lambda}+f_3\frac{1}{\kappa}\right)=g'_2\lambda \left( f_1+f_2\frac{1}{\lambda}+f_3\frac{1}{\kappa}\right)
=g'_3\kappa\left( f_1+f_2\frac{1}{\lambda}+f_3\frac{1}{\kappa}\right), \\
f_1g''_1+f_2g''_2+f_3g''_3=g''_1\left( f_1+f_2\frac{1}{\lambda}+f_3\frac{1}{\kappa}\right)=g''_2\lambda \left( f_1+f_2\frac{1}{\lambda}+f_3\frac{1}{\kappa}\right)
=g''_3\kappa\left( f_1+f_2\frac{1}{\lambda}+f_3\frac{1}{\kappa}\right).
\end{align*}

Let $\dfrac{w}{z}:=f_1+f_2\dfrac{1}{\lambda}+f_3\dfrac{1}{\kappa}$ such that $\mathrm{gcd}(w,z)=1$.
Since $\degree (w)-\degree (z)=\degree (f_1)$ and $\degree (f_1) > 0$, $w$ is a nonconstant polynomial. In addition, $\dfrac{g_1}{z} w$ is a polynomial. Since $\mathrm{gcd}(w,z)=1$, the ratio $\dfrac{g_1}{z}$ must be a polynomial. Similarly, $\dfrac{g_2}{z}\lambda$ and $\dfrac{g_3}{z}\kappa$ are polynomials too.

By Theorem \ref{thmPOLY}, there exists ${\gamma} \in  \mathbb{P}^{1}_k$ such that the non-zero polynomial
$w({\gamma})=0$. By Equation \ref{eqqq20}, we obtain
$(m_{1234}^{1234}(\mathbf{M}))({\gamma})=(m_{1234}^{1235}(\mathbf{M}))({\gamma})=(m_{1234}^{1245}(\mathbf{M}))({\gamma})=0$.
Since $(m_{1234}^{1345}(\mathbf{M}))({\gamma})$ and $(m_{1234}^{2345}(\mathbf{M}))({\gamma})$ are already zero, all $4\times 4$ minors of $\mathbf{M}$ are zero at ${\gamma}$. Then we have $\Rank(\mathbf{M})(\gamma)\leq 3$, so $\Rank (\mathbf{D}({\gamma}))\leq 3$, which is a contradiction.

\end{proof}

 Theorem \ref{MainTheorem} proves that Conjecture \ref{Berrinconjecture} holds for $N=8$, and hence this proves Conjecture \ref{ENSONBerrinconjecture} for $N=8$.

\appendix
\section{}
 The aim of this section is to show that every element in $\mathbf{RP}(8)$ can be obtained from a unique maximal element by a sequence of
 moves of type $\RN{3}$. When $N=8$, the total number of elements of $\mathbf{RP}(8)$ is $105$ and $14$ of them are maximal. The Hasse diagram of $\mathbf{RP}(8)$ with representative permutations was computed using $\mathrm{GAP}\, 4.8.3$, and is as follows:

\vfill
\

\begin{figure}[H]
    \centering
    \resizebox{1\textwidth}{!}{%
\begin{tikzpicture}[scale=2.5]
  \node (L1-0) at (-6.8,16) {$L1$};
  \node (L1-1) at (-6,16) {$\tiny{\left(\begin{array}{cccc} 1& 3& 5&7 \\2& 4& 6& 8\end{array}\right)}$};
  \node (L1-2) at (-5,16) {$\tiny{\left(\begin{array}{cccc} 1& 3& 5&6 \\2& 4& 8& 7\end{array}\right)}$};
  \node (L1-3) at (-4,16) {$\tiny{\left(\begin{array}{cccc} 1& 3& 4&7 \\2& 6& 5& 8\end{array}\right)}$};
  \node (L1-4) at (-3,16) {$\tiny{\left(\begin{array}{cccc} 1& 3& 4&6 \\2& 8& 5& 7\end{array}\right)}$};
  \node (L1-5) at (-2,16) {$\tiny{\left(\begin{array}{cccc} 1& 3& 4&5 \\2& 8& 7& 6\end{array}\right)}$};
  \node (L1-6) at (-1,16) {$\tiny{\left(\begin{array}{cccc} 1& 2& 5&7 \\4& 3& 6& 8\end{array}\right)}$};
  \node (L1-7) at (0,16)  {$\tiny{\left(\begin{array}{cccc} 1& 2& 5&6 \\4& 3& 8& 7\end{array}\right)}$};
  \node (L1-8) at (1,16)  {$\tiny{\left(\begin{array}{cccc} 1& 2& 4&7 \\6& 3& 5& 8\end{array}\right)}$};
  \node (L1-9) at (2,16)  {$\tiny{\left(\begin{array}{cccc} 1& 2& 3&7 \\6& 5& 4& 8\end{array}\right)}$};
  \node (L1-10) at (3,16) {$\tiny{\left(\begin{array}{cccc} 1& 2& 4&6 \\8& 3& 5& 7\end{array}\right)}$};
  \node (L1-11) at (4,16) {$\tiny{\left(\begin{array}{cccc} 1& 2& 4&5 \\8& 3& 7& 6\end{array}\right)}$};
  \node (L1-12) at (5,16) {$\tiny{\left(\begin{array}{cccc} 1& 2& 3&6 \\8& 5& 4& 7\end{array}\right)}$};
  \node (L1-13) at (6,16) {$\tiny{\left(\begin{array}{cccc} 1& 2& 3&5 \\8& 7& 4& 6\end{array}\right)}$};
  \node (L1-14) at (7,16) {$\tiny{\left(\begin{array}{cccc} 1& 2& 3&4 \\8& 7& 6& 5\end{array}\right)}$};
  \node (L2-0) at (-6.8,13) {$L2$};
  \node (L2-1) at (-6,13) {$1$};
  \node (L2-2) at (-5.5,13) {$2$};
  \node (L2-3) at (-5,13) {$3$};
  \node (L2-4) at (-4.5,13) {$4$};
  \node (L2-5) at (-4,13) {$5$};
  \node (L2-6) at (-3.5,13) {$6$};
  \node (L2-7) at (-3,13) {$7$};
  \node (L2-8) at (-2.5,13) {$8$};
  \node (L2-9) at (-2,13) {$9$};
  \node (L2-10) at (-1.5,13) {$10$};
  \node (L2-11) at (-1,13) {$11$};
  \node (L2-12) at (-0.5,13) {$12$};
  \node (L2-13) at (0,13) {$13$};
  \node (L2-14) at (0.5,13) {$14$};
  \node (L2-15) at (1,13) {$15$};
  \node (L2-16) at (1.5,13) {$16$};
  \node (L2-17) at (2,13) {$17$};
  \node (L2-18) at (2.5,13) {$18$};
  \node (L2-19) at (3,13) {$19$};
  \node (L2-20) at (3.5,13) {$20$};
  \node (L2-21) at (4,13) {$21$};
  \node (L2-22) at (4.5,13) {$22$};
  \node (L2-23) at (5,13) {$23$};
  \node (L2-24) at (5.5,13) {$24$};
  \node (L2-25) at (6,13) {$25$};
  \node (L2-26) at (6.5,13) {$26$};
  \node (L2-27) at (7,13) {$27$};
  \node (L2-28) at (7.5,13) {$28$};
  \node (L3-0) at (-6.8,10) {$L3$};
  \node (L3-1) at (-6,10) {$1$};
  \node (L3-2) at (-5.5,10) {$2$};
  \node (L3-3) at (-5,10) {$3$};
  \node (L3-4) at (-4.5,10) {$4$};
  \node (L3-5) at (-4,10) {$5$};
  \node (L3-6) at (-3.5,10) {$6$};
  \node (L3-7) at (-3,10) {$7$};
  \node (L3-8) at (-2.5,10) {$8$};
  \node (L3-9) at (-2,10) {$9$};
  \node (L3-10) at (-1.5,10) {$10$};
  \node (L3-11) at (-1,10) {$11$};
  \node (L3-12) at (-0.5,10) {$12$};
  \node (L3-13) at (0,10) {$13$};
  \node (L3-14) at (0.5,10) {$14$};
  \node (L3-15) at (1,10) {$15$};
  \node (L3-16) at (1.5,10) {$16$};
  \node (L3-17) at (2,10) {$17$};
  \node (L3-18) at (2.5,10) {$18$};
  \node (L3-19) at (3,10) {$19$};
  \node (L3-20) at (3.5,10) {$20$};
  \node (L3-21) at (4,10) {$21$};
  \node (L3-22) at (4.5,10) {$22$};
  \node (L3-23) at (5,10) {$23$};
  \node (L3-24) at (5.5,10) {$24$};
  \node (L3-25) at (6,10) {$25$};
  \node (L3-26) at (6.5,10) {$26$};
  \node (L3-27) at (7,10) {$27$};
  \node (L3-28) at (7.5,10) {$28$};
  \node (L4-0) at (-6.8,7) {$L4$};
  \node (L4-1) at (-6,7) {$1$};
  \node (L4-2) at (-5.3,7) {$2$};
  \node (L4-3) at (-4.7,7) {$3$};
  \node (L4-4) at (-4,7) {$4$};
  \node (L4-5) at (-3.3,7) {$5$};
  \node (L4-6) at (-2.7,7) {$6$};
  \node (L4-7) at (-2,7) {$7$};
  \node (L4-8) at (-1.3,7) {$8$};
  \node (L4-9) at (-0.7,7) {$9$};
  \node (L4-10) at (0,7) {$10$};
  \node (L4-11) at (0.7,7) {$11$};
  \node (L4-12) at (1.4,7) {$12$};
  \node (L4-13) at (2.1,7) {$13$};
  \node (L4-14) at (2.8,7) {$14$};
  \node (L4-15) at (3.5,7) {$15$};
  \node (L4-16) at (4.2,7) {$16$};
  \node (L4-17) at (4.9,7) {$17$};
  \node (L4-18) at (5.6,7) {$18$};
  \node (L4-19) at (6.3,7) {$19$};
  \node (L4-20) at (7,7) {$20$};
  \node (L5-0) at (-6.8,4) {$L5$};
  \node (L5-1) at (-6,4)   {$\tiny{1=\left(\begin{array}{cccc} 1 & 2 & 4 &5 \\3 & 6 & 7 & 8\end{array}\right)}$};
  \node (L5-2) at (-4.5,4) {$\tiny{2=\left(\begin{array}{cccc} 1 & 2 & 3 &5 \\6 & 4 & 7 & 8\end{array}\right)}$};
  \node (L5-3) at (-3,4)   {$\tiny{3=\left(\begin{array}{cccc} 1 & 2 & 3 &6 \\4 & 5 & 7 & 8\end{array}\right)}$};
  \node (L5-4) at (-1.5,4) {$\tiny{4=\left(\begin{array}{cccc} 1 & 2 & 3 &5 \\4 & 7 & 6 & 8\end{array}\right)}$};
  \node (L5-5) at (0,4)    {$\tiny{5=\left(\begin{array}{cccc} 1 & 2 & 3 &5 \\4 & 6 & 8 & 7\end{array}\right)}$};
  \node (L5-6) at (1.5,4)  {$\tiny{6=\left(\begin{array}{cccc} 1 & 2 & 3 &4 \\6 & 7 & 5 & 8\end{array}\right)}$};
  \node (L5-7) at (3,4)    {$\tiny{7=\left(\begin{array}{cccc} 1 & 2 & 3 &4 \\7 & 5 & 6 & 8\end{array}\right)}$};
  \node (L5-8) at (4.5,4)  {$\tiny{8=\left(\begin{array}{cccc} 1 & 2 & 3 &4 \\6 & 5 & 8 & 7\end{array}\right)}$};
  \node (L5-9) at (6,4)    {$\tiny{9=\left(\begin{array}{cccc} 1 & 2 & 3 &4 \\5 & 7 & 8 & 6\end{array}\right)}$};
  \node (L5-10) at (7.5,4) {$\tiny{10=\left(\begin{array}{cccc} 1 & 2 & 3 &4 \\5 & 8 & 6 & 7\end{array}\right)}$};
  \node (L6-0) at (-6.8,1.5) {$L6$};
  \node (L6-1) at (-4,1.5)   {$\tiny{1=\left(\begin{array}{cccc} 1 & 2 & 3 &4 \\4 & 6 & 7 & 8\end{array}\right)}$};
  \node (L6-2) at (-1,1.5)   {$\tiny{2=\left(\begin{array}{cccc} 1 & 2 & 3 &4 \\6 & 5 & 7 & 8\end{array}\right)}$};
  \node (L6-3) at (2,1.5)    {$\tiny{3=\left(\begin{array}{cccc} 1 & 2 & 3 &4 \\5 & 7 & 6 & 8\end{array}\right)}$};
  \node (L6-4) at (5,1.5)    {$\tiny{4=\left(\begin{array}{cccc} 1 & 2 & 3 &4 \\5 & 6 & 8 & 7\end{array}\right)}$};
  \node (L7-0) at (-6.8,-0.7) {$L7$};
  \node (L7-one) at (0,-0.7) {$\tiny{1=\left(\begin{array}{cccc} 1 & 2 & 3 &4 \\5 & 6 & 7 &8 \end{array}\right)}$};
  \draw (L7-one) -- (L6-2);
  \draw (L7-one) -- (L6-3);
  \draw (L7-one) -- (L6-4);
  \draw (L6-1) -- (L5-2);\draw (L6-1) -- (L5-4); \draw (L6-1) -- (L5-5);\draw (L6-2) -- (L5-6);\draw (L6-2) -- (L5-8);
  \draw (L6-3) -- (L5-7);\draw (L6-3) -- (L5-9);
  \draw (L6-4) -- (L5-8);\draw (L6-4) -- (L5-10);
  \draw (L5-1) -- (L4-3);\draw (L5-1) -- (L4-5); \draw (L5-1) -- (L4-6);
  \draw (L5-2) -- (L4-7);
  \draw (L5-3) -- (L4-8); \draw (L5-3) -- (L4-11);\draw (L5-3) -- (L4-12);
  \draw (L5-4) -- (L4-7); \draw (L5-4) -- (L4-9);
  \draw (L5-5) -- (L4-10); \draw (L5-5) -- (L4-13);\draw (L5-5) -- (L4-14);
  \draw (L5-6) -- (L4-15); \draw (L5-6) -- (L4-17);
  \draw (L5-7) -- (L4-16);
  \draw (L5-8) -- (L4-15); \draw (L5-8) -- (L4-17);\draw (L5-8) -- (L4-20);
  \draw (L5-9) -- (L4-18); \draw (L5-9) -- (L4-19);
  \draw (L5-10) -- (L4-19); \draw (L5-10) -- (L4-20);
  \draw (L4-1) -- (L3-2); \draw (L4-1) -- (L3-2);
  \draw (L4-2) -- (L3-6); \draw (L4-2) -- (L3-7);
  \draw (L4-3) -- (L3-9); \draw (L4-3) -- (L3-10);
  \draw (L4-4) -- (L3-8); \draw (L4-4) -- (L3-11);\draw (L4-4) -- (L3-12);
  \draw (L4-5) -- (L3-9); \draw (L4-5) -- (L3-13);
  \draw (L4-6) -- (L3-13); \draw (L4-6) -- (L3-14);
  \draw (L4-7) -- (L3-16);
  \draw (L4-8) -- (L3-15);
  \draw (L4-9) -- (L3-16);
  \draw (L4-10) -- (L3-17);
  \draw (L4-11) -- (L3-18);\draw (L4-11) -- (L3-21);
  \draw (L4-12) -- (L3-19); \draw (L4-12) -- (L3-21);
  \draw (L4-13) -- (L3-20);  \draw (L4-13) -- (L3-22);
  \draw (L4-14) -- (L3-22); \draw (L4-14) -- (L3-26);
  \draw (L4-15) -- (L3-23); \draw (L4-15) -- (L3-24);
  \draw (L4-16) -- (L3-23);
  \draw (L4-17) -- (L3-24); \draw (L4-17) -- (L3-25);  \draw (L4-17) -- (L3-27);
  \draw (L4-18) -- (L3-23); \draw (L4-18) -- (L3-25);\draw (L4-18) -- (L3-28);
  \draw (L4-19) -- (L3-25);
  \draw (L4-20) -- (L3-27); \draw (L4-20) -- (L3-28);
  \draw (L3-1) -- (L2-2); \draw (L3-1) -- (L2-4);
  \draw (L3-2) -- (L2-5);
  \draw (L3-3) -- (L2-6);
  \draw (L3-4) -- (L2-7); \draw (L3-4) -- (L2-10);
  \draw (L3-5) -- (L2-8); \draw (L3-5) -- (L2-11);
  \draw (L3-6) -- (L2-12);
  \draw (L3-7) -- (L2-12); \draw (L3-7) -- (L2-13);
  \draw (L3-8) -- (L2-14);
  \draw (L3-9) -- (L2-16);
  \draw (L3-10) -- (L2-16);
  \draw (L3-11) -- (L2-17);
  \draw (L3-12) -- (L2-15); \draw (L3-12) -- (L2-17);
  \draw (L3-13) -- (L2-18);
  \draw (L3-14) -- (L2-23);
  \draw (L3-15) -- (L2-19); \draw (L3-15) -- (L2-20);
  \draw (L3-16) -- (L2-21); \draw (L3-16) -- (L2-24);
  \draw (L3-17) -- (L2-21);
  \draw (L3-18) -- (L2-19); \draw (L3-18) -- (L2-25);
  \draw (L3-19) -- (L2-20);
  \draw (L3-20) -- (L2-21);
  \draw (L3-21) -- (L2-25);
  \draw (L3-22) -- (L2-26);
  \draw (L3-23) -- (L2-27);
  \draw (L3-24) -- (L2-27);
  \draw (L3-25) -- (L2-22);
  \draw (L3-26) -- (L2-26);
  \draw (L3-27) -- (L2-27); \draw (L3-27) -- (L2-28);
  \draw (L3-28) -- (L2-28);
  \draw (L2-1) -- (L1-2); \draw (L2-2) -- (L1-4);
  \draw (L2-3) -- (L1-3); \draw (L2-4) -- (L1-4);
  \draw (L2-5) -- (L1-5); \draw (L2-6) -- (L1-5);
  \draw (L2-7) -- (L1-7); \draw (L2-8) -- (L1-8);
  \draw (L2-9) -- (L1-6);\draw (L2-10) -- (L1-7);
  \draw (L2-11) -- (L1-8);\draw (L2-12) -- (L1-9);
  \draw (L2-13) -- (L1-9);
  \draw (L2-14) -- (L1-10);
  \draw (L2-15) -- (L1-10);
  \draw (L2-16) -- (L1-11);
  \draw (L2-17) -- (L1-10);
  \draw (L2-18) -- (L1-11);
  \draw (L2-19) -- (L1-12);
  \draw (L2-20) -- (L1-12);
  \draw (L2-21) -- (L1-13);
  \draw (L2-22) -- (L1-14);
  \draw (L2-23) -- (L1-11);
  \draw (L2-24) -- (L1-13);
  \draw (L2-25) -- (L1-12);
  \draw (L2-26) -- (L1-13);
  \draw (L2-27) -- (L1-14);
  \draw (L2-28) -- (L1-14);
  \end{tikzpicture}
  }%
    \caption{The Hasse diagram of $\mathbf{RP}(8)$}
    \label{fig:1}
\end{figure}

\begin{center}
{Here we have abbreviated the entries according to the following table:}
\begin{tabular}{ | m{1cm} | m{12cm} | }
\hline
Level& Involutions \\
\hline
$1$ & $1=(12)(34)(56)(78) , 2=(12)(34)(58)(67), 3=(12)(36)(45)(78)$,
$ 4=(12)(38)(45)(67),  5=(12)(38)(47)(56) , 6=(14)(23)(56)(78)$,
$ 7=(14)(23)(58)(67) , 8=(16)(23)(45)(78) , 9=(16)(25)(34)(78)$,
$ {10=(18)(23)(45)(67) ,\;  11=(18)(23)(47)(56) ,\;  12=(18)(25)(34)(67)}$,
$ 13=(18)(27)(34)(56) , \; 14=(18)(27)(36)(45)$\\
\hline
2 & ${1=(12)(34)(57)(68) , 2=(12)(37)(45)(68) , 3=(12)(35)(46)(78)}$,
$ 4=(12)(35)(48)(67), 5=(12)(37)(48)(56), 6=(12)(38)(46)(57)$,
$ 7=(14)(23)(57)(68) , 8=(15)(23)(46)(78), 9=(13)(24)(56)(78)$,
$ {10=(13)(24)(58)(67) , 11=(13)(26)(45)(78), 12=(15)(26)(34)(78)}$,
$ {13=(16)(24)(35)(78) , 14=(17)(23)(45)(68), 15=(15)(23)(48)(67)}$,
$ {16=(17)(23)(48)(56), 17=(13)(28)(45)(67), 18=(13)(28)(47)(56)}$,
$ {19=(17)(25)(34)(68), 20=(15)(28)(34)(67), 21=(17)(28)(34)(56)}$,
$ {22=(17)(28)(36)(45) , 23=(18)(23)(46)(57), 24=(18)(26)(34)(57)}$,
$ {25=(18)(24)(35)(67) , 26=(18)(24)(37)(56), 27=(18)(26)(37)(45)}$,
$ {28=(18)(27)(35)(46)}$ \\
\hline
$3$&${1=(12)(35)(47)(68), 2=(12)(37)(46)(58), 3=(12)(36)(48)(57)}$,
$ 4=(13)(24)(57)(68), 5=(13)(25)(46)(78), 6=(15)(24)(36)(78)$,
$ 7=(14)(26)(35)(78), 8=(15)(23)(47)(68), 9=(17)(23)(46)(58)$,
$ {10=(16)(23)(48)(57), 11=(13)(27)(45)(68), 12=(13)(25)(48)(67)}$,
$ {13=(13)(27)(48)(56), 14=(13)(28)(46)(57), 15=(15)(27)(34)(68)}$,
$ { 16=(17)(26)(34)(58), 17=(16)(28)(34)(57), 18=(17)(24)(35)(68)}$,
$ {19=(15)(24)(38)(67), 20=(17)(24)(38)(56), 21=(14)(28)(35)(67)}$,
$ { 22=(14)(28)(37)(56), 23=(17)(26)(38)(45), 24=(16)(28)(37)(45)}$,
$ { 25=(17)(28)(35)(46), 26=(18)(24)(36)(57), 27=(18)(26)(35)(47)}$,
$ {28=(18)(25)(37)(46)}$\\
\hline
$4$& ${1=(12)(36)(47)(58), 2=(14)(25)(36)(78), 3=(16)(23)(47)(58)}$,
$ { 4=(13)(25)(47)(68), 5=(13)(27)(46)(58), 6=(13)(26)(48)(57)}$,
$ {7=(16)(27)(34)(58), 8=(15)(24)(37)(68), 9=(17)(24)(36)(58)}$,
$ {10=(16)(24)(38)(57), 11=(14)(27)(35)(68),  12=(14)(25)(38)(67)}$,
$ { 13=(14)(27)(38)(56), 14=(14)(28)(36)(57), 15=(16)(27)(38)(45)}$,
$ { 16=(17)(26)(35)(48), 17=(16)(28)(35)(47), 18=(17)(25)(38)(46)}$,
$ { 19=(15)(28)(37)(46), 20=(18)(25)(36)(47)}$\\
\hline
$5$&${1=(13)(26)(47)(58), 2=(16)(24)(37)(58), 3=(14)(25)(37)(68)}$,
$ { 4=(14)(27)(36)(58), 5=(14)(26)(38)(57), 6=(16)(27)(35)(48)}$,
$ { 7=(17)(25)(36)(48), 8=(16)(25)(38)(47), 9=(15)(27)(38)(46)}$,
$ { 10=(15)(28)(36)(47)}$\\
\hline
$6$& ${1=(14)(26)(37)(58), 2=(16)(25)(37)(48), 3=(15)(27)(36)(48)}$,
${ 4=(15)(26)(38)(47)}$\\
\hline
$7$&$1=(15)(26)(37)(48)$\\
\hline
\end{tabular}
\end{center}
\section{Acknowledgments}
The author was supported by the Fellowship Program for Abroad Studies 2219 by
the Scientific and Technological Research Council of Turkey (T\"{U}B\.{I}TAK). The author is grateful to the Institut f\"{u}r Mathematik at the Universit\"{a}t Augsburg and \mbox{Bernhard Hanke} for their hospitality, and would like to thank Daniel Erman for mentioning that there is a counterexample when $N=12$ and \"{O}zg\"{u}n \"{U}nl\"{u} for many helpful discussions on this paper.


\end{document}